\documentclass[a4paper, 11pt]{article}
\usepackage[T1]{fontenc}
\usepackage[utf8]{inputenc}
\usepackage[all]{xy}
\usepackage{amsfonts}

\usepackage{cite}
\usepackage{enumerate}		
\usepackage{hyperref}		
\usepackage{relsize}
\usepackage{amssymb} 		
\usepackage{mathrsfs} 		
\usepackage{bbm} 		
\usepackage{extarrows} 		
\usepackage{stmaryrd} 		
\usepackage{mathtools}		

\usepackage{amsmath}	
\usepackage{amsthm} 		
\usepackage[nottoc]{tocbibind} 

\theoremstyle{plain}
\newtheorem{thm}{Theorem}
\numberwithin{thm}{section}
\newtheorem{prop}[thm]{Proposition}
\newtheorem{lemma}[thm]{Lemma}

\newtheorem{cor}[thm]{Corollary}

\theoremstyle{definition}
\newtheorem{defn}{Definition}
\numberwithin{defn}{section}

\newtheorem{ex}{Example}
\numberwithin{ex}{section}

\numberwithin{nota}{section}

\newtheorem{remark}{Remark}
\newtheorem{remarks}[remark]{Remarks}
\numberwithin{remark}{section}

\usepackage{todonotes}
\setuptodonotes{color=blue!10!white}
\usepackage{tikz}
\usepackage{tikz-cd} 
\usetikzlibrary{positioning,intersections} 
\usetikzlibrary{decorations.pathmorphing} 
\usetikzlibrary{arrows}

\usepackage{xspace} 
 
\newcommand{\ie}{i.e.\@\xspace}

\newcommand{\ac}{`}

\renewcommand{\epsilon}{\varepsilon}
\renewcommand{\phi}{\varphi}
\renewcommand{\Im}{\operatorname{Im}}

\newcommand{\op}{^{\textup{op}}}

\newcommand{\can}{^{\textup{can}}}

\newcommand{\cartcov}{^{\textup{cart,cov}}}

\newcommand{\pmap}
{\rightharpoondown}

\newcommand{\Hom}{\operatorname{Hom}}

\renewcommand{\lim}{\operatorname{lim}}

\newcommand{\dom}{{\operatorname{dom}}}

\newcommand{\comma}[2]			
{\mbox{$(#1\!\downarrow\!#2)$}}

\newcommand{\sheafify}{\mathrm{a}}	

\newcommand{\CAT}{\mathbf{CAT}}

\newcommand{\Topos}{\mathbf{Topos}}

\newcommand{\Set}{\mathbf{Set}}

\newcommand{\Loc}{\mathbf{Loc}}

\newcommand{\Sh}{\mathbf{Sh}}

\newcommand{\Fib}{{\mathbf{Fib}}}

\newcommand{\Sub}{\mathrm{Sub}}

\newcommand{\dcat}{\mathbb{D}}
\newcommand{\ecat}{\mathbb{E}}

\newcommand{\cbicat}{\mathcal{C}}
\newcommand{\cbciat}{\mathcal{C}}
\newcommand{\dbicat}{\mathcal{D}}

\newcommand{\gbicat}{\mathcal{G}}

\newcommand{\Etopos}{{\cal E}}
\newcommand{\Ftopos}{{\cal F}}

\newcommand{\mono}{\rightarrowtail}

\makeatletter
\newbox\xrat@below
\newbox\xrat@above
\newcommand{\xrightarrowtail}[2][]{%
	\setbox\xrat@below=\hbox{\ensuremath{\scriptstyle #1}}%
	\setbox\xrat@above=\hbox{\ensuremath{\scriptstyle #2}}%
	\pgfmathsetlengthmacro{\xrat@len}{max(\wd\xrat@below,\wd\xrat@above)+.6em}%
	\mathrel{\tikz [>->,baseline=-.75ex]
		\draw (0,0) -- node[below=-2pt] {\box\xrat@below}
		node[above=-2pt] {\box\xrat@above}
		(\xrat@len,0) ;}}
\makeatother

\tikzset{Rightarrow/.style={double equal sign distance,>={Implies},->},
	triple/.style={-,preaction={draw,Rightarrow}}}

\begin{document}

\title{Fibred sites and existential toposes}

\author{Olivia Caramello}

\date{22 December 2022}

\maketitle

\begin{abstract}
	In the context of relative topos theory via stacks, we introduce the notion of existential fibred site and of existential topos of such a site. These notions allow us to develop relative topos theory in a way which naturally generalizes the construction of toposes of sheaves on locales, and also provides a framework for investigating the connections between Grothendieck toposes as built from sites and elementary toposes as built from triposes. Then we focus on fibred preorder sites and establish a fibred generalisation of the ideal-completion of a preorder site. Lastly, we provide an explicit description of the hyperconnected-localic factorization of a geometric morphism in terms of internal locales.    
\end{abstract}

\tableofcontents

\section{Introduction}

The aim of this paper is to present a way for representing relative toposes, in the sense of \cite{CaramelloZanfa}, which naturally generalizes the construction of the topos of sheaves on a locale, and which is particularly effective for describing in a simple way the morphisms between relative toposes.

Recall that, given locales $L$ and $L'$, the morphisms $\Sh(L)\to \Sh(L')$ correspond exactly to the locale homomorphisms $L\to L'$; this actually generalizes to locales internal to an arbitrary Grothendieck topos, as shown in section \ref{sec:localicmorphisms}. 

Our representation of relative toposes is based on the concept of \emph{existential fibred site}, introduced in section \ref{sec:existentialfibredsites} after providing, in the previous sections, the necessary preliminaries on relative topos theory. By using this notion, we shall be able to describe the morphisms between two relative toposes as morphisms between the associated existential fibred sites. This result arises as a byproduct of a relative version of Diaconescu's theorem, formulated in the setting of indexed categories and stacks, which we establish in section \ref{sec:relativeDiaconescu}. The notion of existential fibred site also allows us to obtain a site-theoretic characterization of internal locales to toposes of sheaves on arbitrary (non-necessarily cartesian) sites. 

The construction of \emph{existential toposes}, i.e. of (relative) toposes of sheaves on existential fibred sites, combines and generalizes, in a sense, the construction of Grothendieck toposes from sites with the fibrational construction of elementary toposes from triposes (cf. \cite{tripostheory}). Indeed, as it will be shown in future work on the subject, the tripos-to-topos construction can be naturally understood in terms of existential fibred sites. 
 
In section \ref{sec:completionfibredpreordersite} we focus our attention on fibred preorder sites, and introduce a fibred generalization of the ideal-completion of a preorder site. 

Lastly, in section \ref{sec:hyperconn}, we investigate the hyperconnected-localic factorization of a geometric morphism from the point of view of internal locales, providing an explicit description of such locales in terms of morphisms of sites inducing the given morphism.

\section{Preliminaries}\label{sec:preliminaries}

\subsection*{Terminology}

Given a (small-generated) site $({\cal C}, J)$, we shall denote by $l_{J}$, or simply by $l$ (when the topology $J$ can be unambigously inferred from the context), the canonical functor ${\cal C}\to \Sh({\cal C}, J)$, and by $a_{J}$ the associated sheaf functor $[{\cal C}^{\textup{op}}, \Set] \to \Sh({\cal C}, J)$.

Given two (small-generated) sites $({\cal C}, J)$ and $({\cal D}, K)$, the geometric morphism $\Sh({\cal D}, K) \to \Sh({\cal C}, J)$ induced by a morphism of sites $F:({\cal C}, J) \to ({\cal D}, K)$ will be denoted by $\Sh(F)$, while the geometric morphism induced by a comorphism of sites $G:({\cal D}, K) \to ({\cal C}, J)$ will be denoted by $C_{G}$. 

Given an object $E$ of a category $\cal E$ and an arrow $f:E'\to E$ in $\cal E$, we shall denote $f$ by $[f:E' \to E]$ when considered as an object of the slice category ${\cal E}\slash E$.

The canonical topology on a Grothendieck topos $\cal E$ will be denoted by $J^{\textup{can}}_{\cal E}$.

The $2$-category of Grothendieck toposes, geometric morphisms and geometric transformations will be denoted by $\Topos$.

The \ac meta-$2$-category' of categories, functors and
natural transformations will be denoted by $\CAT$.

For the notions of weakly dense morphism of sites, of $J$-dense, $J$-faithful and $J$-full functor we refer the reader to \cite{denseness}.

\subsection{Relative sites, relative toposes}\label{subsec:relsitesreltop}

Let us first recall the theory of relative sites and relative toposes as developed in \cite{CaramelloZanfa}. 
	
Given an indexed category ${\mathbb D}:{\cal C}^{\textup{op}} \to \textup{Cat}$ and a Grothendieck topology $J$ on $\cal C$, we shall denote by 
\[
p_{\mathbb D}:{\cal G}({\mathbb D}) \to {\cal C}
\]
the fibration associated with $\mathbb D$ through the Grothendieck's construction. 
	
Given a Grothendieck topology $J$ on $\cal C$, the \emph{Giraud topology} $J_{\mathbb D}$ (also denoted by $M^{p_{\mathbb C}}_{J}$) on ${\cal G}({\mathbb D})$ is the smallest topology which makes the projection functor $p_{\mathbb D}:{\cal G}({\mathbb D}) \to {\cal C}$ a comorphism of sites to $({\cal C}, J)$.	
	
	\begin{defn}
		Let $({\cal C}, J)$ be a small-generated site. A \emph{relative site}\index{site!relative -} over $({\cal C}, J)$ is a site of the form $({\cal G}({\mathbb D}), J')$, where ${\mathbb D}$ is a $\cal C$-indexed category and $J'$ is a Grothendieck topology on ${\cal G}({\mathbb D})$ containing the Giraud topology $J_{\mathbb D}$. 
		
		Any relative site $({\cal G}({\mathbb D}), J')$ is endowed with the structure comorphism of sites $p_{\mathbb D}:({\cal G}({\mathbb D}), J') \to ({\cal C}, J)$.
	\end{defn}

Given relative sites $p_{\mathbb D}:({\cal G}({\mathbb D}), J') \to ({\cal C}, J)$ and $p_{\mathbb D}':({\cal G}({\mathbb D}'), J'') \to ({\cal C}, J)$ over a site $({\cal C}, J)$, we say that a \emph{morphism of relative sites} over $({\cal C}, J)$ is a morphism of fibrations $p_{\mathbb D}\to p_{{\mathbb D}'}$ which is also a morphism of sites $({\cal G}({\mathbb D}), J')  \to ({\cal G}({\mathbb D}'), J'')$.

We shall denote by $\textup{\bf MorRelSites}_{({\cal C}, J)}(p_{\mathbb D}, p_{{\mathbb D}'})$ the category of morphisms of relative sites $p_{\mathbb D} \to p_{{\mathbb D}'}$ over $({\cal C}, J)$ and natural transformations between them.

\begin{defn}
Let $({\cal C}, J)$ be a small-generated site. A \emph{relative topos} over $\Sh({\cal C}, J)$ is a Grothendieck topos $\cal E$, together with a geometric morphism $p:{\cal E}\to \Sh({\cal C}, J)$.
\end{defn}
	
\begin{thm}[Theorem 8.2.1 \cite{CaramelloZanfa}]\label{thm:characterizationrelativetoposes}
		Let $({\cal C}, J)$ be a small-generated site. Then any relative site over $({\cal C}, J)$ yields a relative topos over $\Sh({\cal C}, J)$; more precisely, any relative site $$p_{{\mathbb D}}:({\cal G}({\mathbb D}), J')\to ({\cal C}, J)$$ induces the relative topos
		\[
		C_{p_{\mathbb D}}:\Sh({\cal G}({\mathbb D}), J')\to \Sh({\cal C}, J),
		\]
		where $C_{p_{\mathbb D}}$ is the geometric morphism induced by $p_{\mathbb D}$, regarded as a comorphism of sites $({\cal G}({\mathbb D}), K) \to ({\cal C}, J)$.
		
		Conversely, any relative topos $f:{\cal E}\to \Sh({\cal C}, J)$ is of the form $C_{p_{\mathbb D}}$ for some relative site $p_{{\mathbb D}}:({\cal G}({\mathbb D}), J')\to ({\cal C}, J)$ (for instance, one can take $p_{{\mathbb D}}$ to be the \emph{canonical relative site of $f$}, in the sense of section \ref{sec:canonicalrelativesite} below).
\end{thm}

A morphism of relative toposes $(f:{\cal F}\to {\cal E})\to (f':{\cal F}' \to {\cal E})$ is a geometric morphism $g:{\cal F}\to {\cal F}'$, together with a natural isomorphism $f'\circ g\cong f$.

\subsection{The canonical relative site of a geometric morphism}\label{sec:canonicalrelativesite}

\begin{defn}
	Let $f:{\cal F}\to {\cal E}$ be a geometric morphism. The \emph{relative topology of $f$} is the Grothendieck topology on the category $(1_{\cal F}\downarrow f^{\ast})$ induced by the canonical topology on $\cal F$ via the projection functor $\pi_{\cal F}:(1_{\cal F}\downarrow f^{\ast}) \to {\cal F}$. 	
\end{defn}	

\begin{thm}[cf. Theorem 3.16 \cite{denseness}]
	Let $f:{\cal F}\to {\cal E}$ be a geometric morphism. Then the canonical projection functor 
	\[
	\pi_{\cal E}:(1_{\cal F}\downarrow f^{\ast}) \to {\cal E}
	\]
	is a comorphism of sites 
	\[
	((1_{\cal F}\downarrow f^{\ast}), J_{f}) \to ({\cal E}, J^{\textup{can}}_{\cal E})
	\]
	such that $f=C_{\pi_{\cal E}}$.
\end{thm}

The functor $\pi_{\cal E}:(1_{\cal F}\downarrow f^{\ast}) \to {\cal E}$ is actually a \emph{stack} on $\cal E$, which we call the \emph{canonical stack of $f$}: from an indexed point of view, the corresponding $\cal E$-indexed category ${\mathbb I}_{f}:{\cal E}^{\textup{op}}\to \CAT$ sends any object $E$ of $\cal E$ to the topos ${\cal F}\slash f^{\ast}(E)$ and any arrow $u:E'\to E$ to the pullback functor $u^{\ast}:{\cal F}\slash f^{\ast}(E)\to {\cal F}\slash f^{\ast}(E')$.
 
The comorphism of sites $\pi_{\cal E}:((1_{\cal F}\downarrow f^{\ast}), J_{f}) \to ({\cal E}, J^{\textup{can}}_{\cal E})$ is called the \emph{canonical relative site} of $f$. 

More generally, given a morphism of small-generated sites $A:(\cbicat,J)\to (\dbicat,K)$, the \emph{relative site of $A$} is the site $((1_{\cal D}\downarrow A), J^{K}_{A})$, where $J^{K}_{A}$ is the Grothendieck topology whose covering families are those which are sent by the canonical projection $(1_{\cal D}\downarrow A) \to {\cal D}$ to $K$-covering families, together with the canonical projection functor $\pi_{A}:(1_{\cal D}\downarrow A) \to {\cal C}$, which is a comorphism of sites $((1_{\cal D}\downarrow A), J^{K}_{A}) \to ({\cal C}, J)$.

\begin{thm}[cf. Theorem 3.16 \cite{denseness}]
	Let $A:(\cbicat,J)\rightarrow (\dbicat,K)$ be a morphism of small-generated sites. Then the geometric morphism $\Sh(A)$ induced by $A$ coincides with the structure geometric morphism $C_{\pi_{A}}$ associated with the relative site $\pi_{A}:((1_{\cal D}\downarrow A), J^{K}_{A})\to ({\cal C}, J)$. 
\end{thm}

In the context of a localic morphism $f$, it will also be convenient to consider the sub-site $((1_{\cal F}\downarrow^{\Sub} f^{\ast}), J_{f}|_{(1_{\cal F}\downarrow^{\Sub} f^{\ast})}$, called the \emph{reduced relative site} of $f$ (cf. Definition \ref{defreducedrelativesite} below) given by the full subcategory of $(1_{\cal F}\downarrow f^{\ast})$ on the objects of the form $(F, E, \alpha:F\to f^{\ast}(E))$, where $\alpha$ is a monomorphism, with the induced Grothendieck topology. The following result is relevant for the study of morphisms between relative localic toposes.

\begin{prop}\label{prop:preservationmono}
	Let $g:[f:{\cal F}\to {\cal E}]\to [f':{\cal F}' \to {\cal E}]$ be a relative geometric morphism. Then the functor
	\[
	g^{\ast}_{\textup{fib}}:((1_{{\cal F}'}\downarrow f'^{\ast}), J_{f'}) \to ((1_{\cal F}\downarrow f^{\ast}), J_{f})
	\]
	restricts to a morphism of sites
	\[
	((1_{{\cal F}'}\downarrow^{\textup{Sub}} f'^{\ast}), J_{f'}|_{(1_{{\cal F}'}\downarrow^{\textup{Sub}} f'^{\ast})}), ((1_{\cal F}\downarrow^{\textup{Sub}} f^{\ast}), J_{f}|_{((1_{\cal F}\downarrow^{\textup{Sub}} f^{\ast})})
	\]
\end{prop}

\begin{proof}
	It suffices to observe that an object $(F', E', \alpha:F'\to f'^{\ast}(E'))$ can be regarded as an arrow $(\alpha, 1):(f'^{\ast}(E'), E', 1_{f'^{\ast}(E')})$ in the category $(1_{{\cal F}'}\downarrow f'^{\ast})$ (and similarly for the relative category of the morphism $f$). Now, if $\alpha$ is monic then the corresponding arrow $(\alpha, 1)$ if also monic, and hence is sent by the functor $g^{\ast}_{\textup{fib}}$ to a monic arrow towards the identity, which in turn identifies with an object of the subcategory $(1_{\cal F} \downarrow^{\textup{Sub}} f^{\ast})$, as required. 
\end{proof}

\subsection{The canonical functor for relative toposes}

A relative topos is in particular, by definition, a topos of the form $C_{G}:\Sh({\cal D}, K)\to \Sh({\cal C}, J)$, where $G:({\cal D}, K)\to \Sh({\cal C}, J)$ is a comorphism of sites.

Note that, any \ac absolute' topos, that is, any Grothendieck topos $\Sh({\cal D}, K)$ considered as a topos over $\Set$, can be naturally represented in this way, by taking as site of definition for $\Set$ the site $(\textbf{1}, T)$ consisting of the one-object, one-arrow category $\textbf{1}$ and the trivial Grothendieck topology $T$ on it, and as comorphism of sites $({\cal D}, K) \to (\textbf{1}, T)$ the unique functor ${\cal D} \to \textbf{1}$.

It is natural to wonder what is the analogue, for a general relative topos $C_{G}$, of the canonical functor ${\cal D}\to \Sh({\cal D}, K)$ for \ac absolute' topos $\Sh({\cal D}, K) \to \Set$. We naturally expect that such a canonical functor should be compatible with the canonical functors to the base topos $\Sh({\cal C}, K)$. Indeed, the functor that we shall define will satisfy this requirement.

\begin{prop}\label{propbidense}
	Let $G:({\cal D}, K)\to ({\cal C}, J)$ be a comorphism of sites. There is a dense bimorphism of sites 
	\[
	\eta_{G}:({\cal D}, K) \to ((1_{\Sh({\cal D}, K)} \downarrow C_{G}^{\ast}), J_{C_{G}})
	\]
	 mapping an object $D$ of $\cal D$ to the triplet $(l_K(D), l_J(G(D)), x_{D}:l_K(D)\to C_{G}^{\ast}(l_J(G(D))))\cong \sheafify_{K}(\cbicat(G(-), G(D)))$, where $x_{D}$ is the image under $\sheafify_{K}$ of the arrow $y_{\cal D}(D)\to \cbicat(G(-), G(D))$ corresponding to the identity on $G(D)$ via the Yoneda embedding. The action of $\eta_{G}$ on arrows is straightforward.
	 
	 \[\begin{tikzcd}
	 	{({\cal D}, K)} &&& {((1_{\textup{\bf Sh}({\cal D}, K)} \downarrow C_{G}^{\ast}), J_{C_{G}})} \\
	 	\\
	 	& {({\cal C}, J)} & {(\textup{\bf Sh}({\cal C}, J), J^{\textup{can}}_{\textup{\bf Sh}({\cal C}, J)})}
	 	\arrow["{\eta_{G}}", from=1-1, to=1-4]
	 	\arrow["{\pi_{\textup{\bf Sh}({\cal C}, J)}}", from=1-4, to=3-3]
	 	\arrow["{l_{J}}", from=3-2, to=3-3]
	 	\arrow["G"', from=1-1, to=3-2]
	 \end{tikzcd}\]

	If $G$ is faithful then $\eta_{G}$ takes values in the subcategory $(1_{\Sh({\cal D}, K)} \downarrow^{\textup{Sub}} C_{G}^{\ast})$.  
\end{prop}

\begin{proof}
The functor
\[
\xi_{G}: (G\downarrow 1_{\cal C}) \to (1_{\Sh({\cal D}, K)} \downarrow C_{G}^{\ast})
\]
sending an object $(d, c, \alpha:G(d)\to c)$ of the category $(G\downarrow 1_{\cal C})$ to the object $(l'(d), l(c), l'(d)\to C_{G}^{\ast}(l(c))\cong a_{K}(G^{\ast}(y_{\cal D}(d))))=a_{K}(\Hom_{\cal C}(G(-), c))$ determined by it via the Yoneda lemma, is a dense bimorphism of sites
$((G\downarrow 1_{\cal C}), \overline{K}) \to ((1_{\Sh({\cal D}, K)} \downarrow C_{G}^{\ast}), J_{C_{G}})$. Indeed, it is the composite of the functor 
	\[
	z_{G}:(G\downarrow 1_{\cal C}) \to (1_{\hat{\cal D}}\downarrow m_{G})
	\]	 
of Theorem 3.20(ii) \cite{denseness} with the functor $(1_{\hat{\cal D}}\downarrow m_{G}) \to  (1_{\Sh({\cal D}, K)} \downarrow C_{G}^{\ast})$ induced by the sheafication functor $\hat{\cal D}\to \Sh({\cal D}, K)$.
				
Now, the functor $\eta_{G}$ is the composite of $\xi_{G}$ with the functor $j_{G}:({\cal D}, K) \to ((G\downarrow 1_{\cal C}), \overline{K})$ of Theorem 3.18 \cite{denseness}.
	
Therefore, it is a dense bimorphism of sites $({\cal D}, K) \to ((1_{\Sh({\cal D}, K)} \downarrow C_{G}^{\ast}), J_{C_{G}})$,	since composition of dense bimorphisms of sites are bimorphisms of sites.
	
As per the second statement, we can use the characterization of arrows such that their sheafification is monomorphic provided by Lemma 2.1(i) of \cite{denseness}. Call $y_D:y_{\cal D}(D)\to \cbicat(G(-), G(D))$ the arrow such that $\sheafify_K(y_D)=x_D$: then $x_D$ is monic if and only if for every object $E$ of $\dbicat$ and every pair of arrows $x,x':E\rightrightarrows D$ such that $y_D\circ x=y_D\circ x'$, the sieve $S_{x,x'}=\{f:\dom(f)\to E\ |\ x\circ f=x'\circ f\}$ is $K$-covering. Notice however that the condition $y_D\circ x=y_D\circ x'$ is equivalent via the chain of natural equivalences
	\begin{align*}
		[\dbicat\op,\Set]\left(y_\dbicat(E), C_G^*l_J(G(D))\right)&\simeq [\cbicat\op,\Set]\left(y_\cbicat(G(E)), y_\cbicat(G(D))\right)\\
		&\simeq \cbicat(G(E),G(D))
	\end{align*}	
to the equality $G(x)=G(x')$. If $G$ is faithful, this implies $x=x'$, whence the sieve $S_{x,x'}$ is maximal and \textit{a fortiori} $K$-covering.
\end{proof}

\section{Classifying morphisms of relative toposes}\label{sec:relativeDiaconescu}

In this section we shall establish a classification theorem generalizing Giraud's result stating the universal property of the classifying topos of a stack.

Given a cartesian category $\cal C$, we shall say that a pseudofunctor ${\mathbb D}:{\cal C}^{\textup{op}} \to \CAT$ is \emph{cartesian} if it takes values in cartesian categories and its transition morphisms preserve finite limits. The category ${\cal G}({\mathbb D})$ thus has finite limits, which are preserved by the canonical projection functor to $\cal C$ (cf. Proposition 2.4.2 \cite{CaramelloZanfa}).

\begin{lemma}\label{lemma:fib_with_terminals_has_radj}
	Consider a pseudofunctor $\dcat:\cbicat\op\rightarrow\CAT$ such that every fibre has, and every transition morphism preserves, terminal objects: then $p_\dcat:\gbicat(\dcat)\rightarrow \cbicat$ has a right adjoint $\tau_\dcat:\cbicat\rightarrow\gbicat(\dcat)$, defined on objects as
	\[
	\tau_\dcat(X):=(X,1_{\dcat(X)}).
	\]
\end{lemma}
\begin{proof}
	Given an arrow $(y,a):(Y,V)\rightarrow (X,1_{\dcat(X)})$, the second component $a:V\rightarrow \dcat(y)(1_{\dcat(X)})\simeq 1_{\dcat(Y)}$ is unique. This implies that there is a bijection between $\gbicat(\dcat)((Y,V), (X,1_{\dcat(X)}))$ and $\cbicat(Y,X)$, which is natural in both components and thus shows that $p_\dcat\dashv \tau_\dcat$.    
\end{proof}

\begin{lemma}\label{lemma:fibrewisecartesian_implica_cartesian}
	Consider two cartesian pseudofunctors $\dcat,\ecat:\cbicat\op\rightarrow \CAT$ over a cartesian category $\cbicat$, and a pseudonatural transformation $F:\dcat\Rightarrow \ecat$ such that each component $F_X:\dcat(X)\rightarrow \ecat(X)$ preserves finite limits: then the induced morphism of fibrations $\gbicat(F):\gbicat(\dcat)\rightarrow\gbicat(\ecat)$ preserves finite limits.    
	
	In other words, a morphism of fibrations whose action on fibres preserves finite limits is itself a finite-limit-preserving functor.
\end{lemma}
\begin{proof}
	We set $\gbicat(F)=:B$, and for the sake of simplicity we suppose that the structural isomorphism $p_\ecat\circ \gbicat(F)\cong p_\dcat$ and that all the structural isomorphisms of $B$ are identities. 
	The image of an object $(X,U)$ of $\gbicat(\dcat)$ via $B$ is defined as the object $(X, B(X)(U))$ of $\gbicat(\ecat)$, and the image of an arrow $(y,a):(Y,V)\rightarrow (X,U)$ in $\gbicat(\dcat)$ is the couple
	\[
	(Y, B(Y)(V))\xrightarrow{(y,B(y)(a))} (X, B(X)(U)).
	\]
	Our goal is to prove that $B$ preserves terminal objects and pullbacks.
	
	Since every transition morphism of $B$ is limit-preserving, in particular it holds that \[B(1_\cbicat)(1_{\dcat(1_\cbciat})\simeq 1_{\ecat(1_\cbicat)},\]
	and thus the image via $B$ of the terminal object of $\gbicat(\dcat)$, \ie $(1_\cbicat, 1_{\dcat(1_\cbicat)})$, is the terminal object in $\gbicat(\ecat)$.
	
	Concerning pullbacks, we recall that a pullback square
	\[
	\begin{tikzcd}
		{(P,K)} \ar[d, "{(p,h)}"'] \ar[r, "{(q,k)}"] & {(Z,W)} \ar[d, "{(z,b)}"]\\
		{(Y,V)} \ar[r, "{(y,a)}"'] & {(X,U)} \ar[ul, phantom, "\lrcorner" near end]
	\end{tikzcd}
	\] 
	in $\gbicat(\dcat)$ is built from the arrows of the two pullback squares
	\[
	\begin{tikzcd}
		P \ar[d, "p"'] \ar[r,"q"] & Z\ar[d, "z"]\\
		Y\ar[r,"y"'] & X\ar[ul, phantom, "\lrcorner" near end]
	\end{tikzcd}
	\begin{tikzcd}
		K \ar[d, "h"'] \ar[rr, "k"] && \dcat(q)(W) \ar[d, "\dcat(q)(b)"]\\
		\dcat(p)(V) \ar[r, "\dcat(p)(a)"'] & \dcat(p)\dcat(y)(U) \ar[r, "\sim"] & \dcat(q)\dcat(z)(U) \ar[ull, phantom, "\lrcorner" near end]
	\end{tikzcd}
	\]
	built respectively in $\cbicat$ and in $\dcat(P)$. Its image via $B$ is the square
	\[
	\begin{tikzcd}[column sep=15ex]
		{(P,B(P)(K))} \ar[d, "{(p,B(P)(h))}"'] \ar[r, "{(q,B(P)(k))}"] & {(Z,B(Z)(W))} \ar[d, "{(z,B(Z)(b))}"]\\
		{(Y,B(Y)(V))} \ar[r, "{(y,B(Y)(a))}"'] & {(X,B(X)(U)):} 
	\end{tikzcd}
	\]
	it is now sufficient to notice that the square
	\[
	\begin{tikzcd}
		{B(P)(K)} \ar[d, "{B(P)(h)}"'] \ar[rr, "{B(P)(k)}"] && {B(q)(B(Z)(W))} \ar[d, "{B(q)(B(Z)(b))}"]\\
		{B(p)(B(Y)(V))} \ar[r, "{B(p)(B(Y)(a))}"{yshift=1ex}] & {B(p)(B(y)(B(X)(U)))} \ar[r, "\sim"] & {B(p)(B(Z)(\dcat(y)(U))} 
	\end{tikzcd}
	\]
	formed by the second components is still a pullback square: indeed, up to canonical isomorphisms it is the image via $B(P)$ (which preserves pullbacks) of the pullback square above whose sides are $h$, $\dcat(p)(a)$, $\dcat(q)(b)$ and $q$.
\end{proof}

The following result is a generalization to arbitrary (non-necessarily trivial) cartesian relative sites over a cartesian site of Giraud's Corollary 2.5 of \cite{giraud.classifying}, asserting the universal property of the classifying topos of a cartesian stack: 
\begin{thm}\label{thm:RelativeDiaconescuCartesian}
	Let $({\cal C}, J)$ be a small-generated site, where $\cal C$ is a cartesian category, ${\mathbb D}:{\cal C}^{\textup{op}} \to \CAT$ a cartesian pseudofunctor, $K$ a Grothendieck topology on ${\cal G}({\mathbb D})$ containing Giraud's topology $J_\dcat$, $A:{\cal C} \to \Ftopos$ a cartesian $J$-continuous functor inducing a geometric morphism $f:\Ftopos\to \Sh({\cal C}, J)$. Then, considering $p_{\mathbb D}$ as a comorphism of sites $({\cal G}({\mathbb D}), K) \to ({\cal C}, J)$, we have an equivalence of categories
	\[
	\Topos/{\Sh({\cal C}, J)}([f], [C_{p_{\mathbb D}}])\simeq \Fib_{\cal C}\cartcov((\gbicat(\dcat), K),(\comma{1_\Ftopos}{A} , J_{f}|_{\comma{1_\Ftopos}{A}})),
	\]
	where $\Fib_{\cal C}\cartcov((\gbicat(\dcat), K),(\comma{1_\Ftopos}{A} , J_{f}|_{\comma{1_\Ftopos}{A}}))$\index{$\Fib\cartcov_\cbicat$} is the category of morphisms of fibrations over $\cal C$ which are cartesian at each fibre and cover-preserving.
\end{thm}

\begin{proof}
	Let us begin with a a relative geometric morphism $g:[f]\to [C_{p_{\mathbb D}}]$: we want to build a cover-preserving cartesian morphism of fibrations $B:\gbicat(\dcat)\to \comma{1_\Ftopos}{A}$ such that $\Sh(B)\cong f$. First of all, we remark that the definition of $B$ is forced by some of the conditions. First of all, the composite
	\[
	\gbicat(\dcat)\xrightarrow{B}\comma{1_\Ftopos}{A}\xrightarrow{\pi_\Ftopos}\Ftopos
	\]
	must be a morphism of sites yielding the geometric morphism $g$: thus in particular $g^*l_K(X,U)\simeq \Sh(\pi_\Ftopos B)^*l_K(X,U)\simeq \pi_\Ftopos B(X,U)$ (as an object of $\Ftopos$): thus the first component of $B(X,U)$ must be the object $g^*l_K(X,U)$ of $\Ftopos$; the second component must be $X$, for we want $B$ to satisfy $\pi_\cbicat\circ B=p_\dcat$. Now, remember from Lemma \ref{lemma:fib_with_terminals_has_radj} that $p_\dcat$ has a right adjoint $\tau_\dcat$: this implies that $$A(X)\simeq f^*l_J(X)\simeq g^* C_{p_\dcat}^* l_J(X)\simeq g^*\Sh(\tau_\dcat)^*l_J(X)\simeq g^*l_K \tau_\dcat(X)=g^*l_K(X,1_{\dcat(X)}).$$
	Since $B(X,U)$ is an arrow $g^*l_K(X,U)\rightarrow A(X)$ in $\Ftopos$, the sensible choice is to set it equal to $g^*l_K(1_X, !)$, where $!:U\to 1_{\dcat(X)}$ denotes the unique arrow to the terminal object of the fibre $\dcat(X)$. This provides us with a definition of a functor
	\[B:\gbicat(\dcat)\to \comma{1_\Ftopos}{A}.\]
	To prove that it is a morphism of fibrations, we need to show that it maps cartesian arrows of $\gbicat(\dcat)$ to cartesian arrows of $\comma{1_\Ftopos}{A}$, \ie pullback squares of $\Ftopos$: but the image of a cartesian arrow $(y,1):(Y,\dcat(y)(U))\to (X,U)$ is the square
	\[
	\begin{tikzcd}[column sep=10ex]
		{B(Y,\dcat(y)(U))=} \ar[d, "{B(y,1)=}"] & {g^*l_K(Y,\dcat(y)(U))} \ar[d,"{g^*l_K(y,1)}"] \ar[r,"{g^*l_K(1_Y,!)}"] & {g^*l_K(Y, 1_{\dcat(Y)})}\ar[d, "{g^*l_K(y,!)}"]\\
		{B(X,U)=} & {g^*l_K(X,U)} \ar[r,"{g^*l_K(1_X,!)}"] & {g^*l_K(X, 1_{\dcat(X)})},
	\end{tikzcd}
	\]
	which is a pullback since it is the image via the finite-limit-preserving functor $g^*l_K$ of a pullback square in $\gbicat(\dcat)$ (cf. the proof of Lemma 2.4.1 \cite{CaramelloZanfa}). Notice that the corresponding pseudonatural trasformation $B:\dcat\Rightarrow \Ftopos/A(-)$ acts fibrewise by mapping $a:U'\to U$ to $g^*l_K(1,a):g^*l_K(X,U')\to g^*l_K(X,U)$: since its behaviour coincides with that of $g^*l_K$, it obviously preserves finite limits (this implies by Lemma \ref{lemma:fibrewisecartesian_implica_cartesian} that the functor $B:\gbicat(\dcat)\rightarrow \comma{1_\Ftopos}{A}$ is cartesian). The functor $B$ is also obviously cover-preserving, for a $K$-covering family is mapped via $g^*l_K$ to a jointly epic family, and thus its image in $\comma{1_\Ftopos}{A}$ is $J_f$-covering.
		
	Conversely, let us start from a morphism of fibrations $B:{\cal G}({\mathbb D}) \to \comma{1_\Ftopos}{A}$ sending $K$-covering families to $J_{f}$-covering families and preserving finite limits at each fibre. By Lemma \ref{lemma:fibrewisecartesian_implica_cartesian}, $B$ is a cartesian functor, and thus it is a morphism of sites $(\gbicat(\dcat), K)\rightarrow ( \comma{1_\Ftopos}{A}, J_{f}|_{(1_{\cal F}\downarrow A)})$: thus, it induces a geometric morphism $\Sh(B):\Ftopos\to \Sh(\dbicat,K)$. Finally, let us denote by $\tau_\dcat$ and $\tau_{\textup{comma}}$ the right adjoints to $p_\dcat:\gbicat(\dcat)\rightarrow \cbicat$ and $\pi_\cbicat:\comma{1_\Ftopos}{A}\rightarrow \cbicat$, again by Lemma \ref{lemma:fib_with_terminals_has_radj}: they are morphisms of sites, since their left adjoints are comorphisms of sites (see \cite[Proposition 3.14(iii)]{denseness}), and it holds that $\Sh(\tau_\dcat)\cong C_{p_\dcat}$ and $\Sh(\tau_{\textup{comma}})\cong C_{\pi_\cbicat}\cong f$. A computation also immediately shows that the triangle of morphisms of sites
	\[
	\begin{tikzcd}
		{(\gbicat(\dcat), K)} \ar[r, "B"] & {(\comma{1_\Ftopos}{A},J_f|_{\comma{1_\Ftopos}{A}})} \\
		{(\cbicat,J)} \ar[u, "\tau_\dcat"] \ar[ur, "\tau_{\textup{comma}}"'] &
	\end{tikzcd}
	\]
	commutes; thus $\Sh(\tau_\dcat)\circ \Sh(B)\cong \Sh(\tau_{\textup{comma}})$, \ie $C_{p_\dcat}\circ \Sh(B)\cong f$ and $\Sh(B)$ is a morphism over $\Sh(\cbicat,J)$, as required.
	
	The fact that these two correspondences yield functors which are quasi-inverse to each other follows from Diaconescu's classical equivalence. 
\end{proof}

\begin{remark}\label{remrelDiaconescu}
	In light of Lemma \ref{lemma:fib_with_terminals_has_radj}, the morphisms on the right-hand side of the equivalence are precisely the morphisms of relative sites $(\gbicat(\dcat), K) \to (\comma{1_\Ftopos}{A} , J_{f}|_{\comma{1_\Ftopos}{A}})$ (in the sense of section \ref{subsec:relsitesreltop}). In fact, as will be shown in future work, the theorem generalizes to the non-cartesian setting, by taking, on the right hand-side, the category of morphisms of relative sites $(\gbicat(\dcat), K) \to (\comma{1_\Ftopos}{A} , J_{f}|_{\comma{1_\Ftopos}{A}})$.     
\end{remark}

\begin{cor}\label{correlativemorphisms}
	Let $f:\Ftopos\to \Etopos$ and $f':\Ftopos' \to \Etopos$ be geometric morphisms towards the same base topos $\Etopos$. Then we have an equivalence of categories
	\[
	\Topos/\Etopos([f], [f'])\simeq \Fib_{\Etopos}\cartcov( 
	(\comma{1_{\Ftopos'}}{f'^*}, J_{f'}), 
	( \comma{1_{\Ftopos}}{f^*}, J_{f})
	),
	\]	
	where $\Fib_{\Etopos}\cartcov( 
	(\comma{1_{\Ftopos'}}{f'^*}, J_{f'}), 
	( \comma{1_{\Ftopos}}{f^*}, J_{f})
	)$ is the category of morphisms of fibrations over $\Etopos$ which are cartesian at each fibre and cover-preserving.
\end{cor}

\begin{proof}
	It suffices to apply Theorem \ref{thm:RelativeDiaconescuCartesian} in the particular case $({\cal C}, J)=(\Etopos, J\can_{\Etopos})$, ${\mathbb D}={\mathbb I}_{f'}$, $A=f^{\ast}$ and $K=J_{f'}$.   
\end{proof}

\begin{cor}\label{cor:fullandfaithfulnessinternallocales}
	Let $\Etopos$ be a Grothendieck topos and $L$, $L'$ internal locales in $\Etopos$. Then we have an equivalence of categories
	\[
	\Topos/\Etopos(\Sh_{\Etopos}(L), \Sh_{\Etopos}(L'))\simeq \Loc(\Etopos)(L, L'),
	\]
	where $\Loc(\Etopos)(L, L')$ is the category of morphisms of internal locales from $L$ to $L'$ in $\Etopos$.	
\end{cor}
\begin{proof}
	Let us denote by $f_{L}$ the canonical structure morphism $\Sh_{\Etopos}(L)\to {\Etopos}$, which is isomorphic to the morphism $C_{p_{L}}$, where $p_{L}$ is regarded as a comorphism of sites $({\cal G}(L), J^{\textup{can}}_{L}) \to  ({\Etopos}, J\can_{\Etopos})$. 
	
	By Theorem \ref{thm:RelativeDiaconescuCartesian}, the category $\Topos/\Etopos(\Sh_{\Etopos}(L), \Sh_{\Etopos}(L'))$ is equivalent to the category of morphisms of fibrations $p_{L'} \to \comma{1_{\Sh_{\Etopos}(L)}}{f_{L}^{\ast}}$
	which send $J\can_{L'}$-covering families to $J_{f_{L}}$-covering sieves and which are cartesian at each fibre. 
	
	By Proposition \ref{prop:preservationmono}, such morphisms of fibrations actually take values in the reduced relative site of the morphism $f_{L}$, namely in the locale $L_{f_{L}}=L$. So the cartesianness at each fibre and the cover-preservation condition amount precisely to requiring having a homomorphisms of internal frames from $L'$ to $L$, that is an internal locale homomorphism $L\to L'$. 
\end{proof}

\section{Localic morphisms and internal locales}\label{sec:localicmorphisms}

In this section we shall investigate the relatives sites of localic geometric morphisms.

Recall that a geometric morphism $f:{\cal F}\to {\cal E}$ is said to be \emph{localic} if every object of $\cal F$ is a subquotient (i.e. a quotient of a subobject) of an object of the form $f^{\ast}(E)$ for $E\in {\cal E}$.

Let $(1_{\cal F}\downarrow^{\Sub} f^{\ast})$ be the the full subcategory of $(1_{\cal F}\downarrow f^{\ast})$	on the objects of the form $(F, E, \alpha:F\to f^{\ast}(E))$, where $\alpha$ is a monomorphism.

\begin{prop}\label{prop:densenessrelsitelocalicmorphism}
	Let $f:{\cal F}\to {\cal E}$ be a localic geometric morphism. Then the full subcategory $(1_{\cal F}\downarrow^{\Sub} f^{\ast})$ of $(1_{\cal F}\downarrow f^{\ast})$	is closed in $(1_{\cal F}\downarrow f^{\ast})$ under finite limits and it is $J_{f}$-dense. 
\end{prop}

\begin{proof}
	The closure of $(1_{\cal F}\downarrow^{\Sub} f^{\ast})$ under finite limits in $(1_{\cal F}\downarrow f^{\ast})$ is immediately seen in light of the concrete construction of such limits in $(1_{\cal F}\downarrow f^{\ast})$, which is preserved by both projections functors to $\cal E$ and $\cal F$.
	
	As far as denseness is concerned, as observed in the proof of Lemma A4.6.3 \cite{elephant}, for any object $F$ of $\cal F$, there is a canonical representation of $F$ as a subquotient of an object in the image of $f^{\ast}$, obtained as follows. Let $\tilde{F}$ be the partial-map representer associated with $f$ (cf. Proposition A2.4.7 \cite{elephant}), with the canonical monomorphism $m_{F}:F\mono \tilde{F}$. Recall that $m_{F}$ is characterized by the following universal property: for any partial map $(m, f):B \pmap F$, there is a unique arrow $\tilde{f}:B\to \tilde{F}$ such that the following square is a pullback:
	\[\begin{tikzcd}
		R & F \\
		B & {\tilde{F}}
		\arrow["{m_{F}}", tail, from=1-2, to=2-2]
		\arrow["m", tail, from=1-1, to=2-1]
		\arrow["{\tilde{f}}", from=2-1, to=2-2]
		\arrow["f", from=1-1, to=1-2]
	\end{tikzcd}\]
	Then, the arrow $w_{F}$ in the pullback square
	\[\begin{tikzcd}
		P & F & \\
		{f^{\ast}(f_{\ast}(\tilde{F}))} & {\tilde{F}}
		\arrow["w_{F}", from=1-1, to=1-2]
		\arrow["z"', tail, from=1-1, to=2-1]
		\arrow["{m_{F}}", tail, from=1-2, to=2-2]
		\arrow["{\epsilon_{\tilde{F}}}", from=2-1, to=2-2]
		\arrow["\lrcorner"{anchor=center, pos=0.125}, draw=none, from=1-1, to=2-2]
	\end{tikzcd}\]
	is an epimorphism (cf. the proof of Lemma A4.6.3 of \cite{elephant}). Indeed, since $f$ is localic, $F$ fits in a diagram of the form
	\[\begin{tikzcd}
		C & F \\
		{f^{\ast}(E)}
		\arrow["n", tail, from=1-1, to=2-1]
		\arrow["u", two heads, from=1-1, to=1-2]
	\end{tikzcd}\]
	where $n$ is a monomorphism and $u$ is an epimorphism. Then, by the universal property of $m_{F}$, we have a pullback square
	\[\begin{tikzcd}
		C & F \\
		{f^{\ast}(E)} & {\tilde{F}}
		\arrow["n", tail, from=1-1, to=2-1]
		\arrow["u", two heads, from=1-1, to=1-2]
		\arrow["{m_{F}}", from=1-2, to=2-2]
		\arrow["{\tilde{u}}", from=2-1, to=2-2]
	\end{tikzcd}\]

	But, by the adjunction $(f^{\ast}\dashv f_{\ast})$, the arrow $\tilde{u}$ factors through $\epsilon_{\tilde{F}}$, whence the epimrophism $u$ factors through $w_{F}$, which is then \emph{a fortiori} an epimorphism.
	
	Now, given an object $(F, E, \alpha:F\to f^{\ast}(E))$ of the category $(1_{\cal F} \downarrow f^{\ast})$,  consider the object
	\[
	(P, f_{\ast}(\tilde{F})\times E, <z, \alpha \circ w>:P \to f^{\ast}(f_{\ast}(\tilde{F})\times E)\cong f^{\ast}(f_{\ast}(\tilde{F}))\times f^{\ast}(E))
	\]
	of the category $(1_{\cal F} \downarrow f^{\ast})$, where the arrow $<z, \alpha \circ w>$ is determined by the arrows $z$ and $\alpha \circ w$ by the universal property of the product and hence is monic (as $z$ is). So this object lies in the subcategory $(1_{\cal F}\downarrow^{\Sub} f^{\ast})$. But it is related to our original object by the arrow 
	\[\begin{tikzcd}
		{(P, f_{\ast}(\tilde{F})\times E, <z, \alpha \circ w>:P \to f^{\ast}(f_{\ast}(\tilde{F}))\times f^{\ast}(E))} \\
		\\
		{(F, E, \alpha:F\to f^{\ast}(E)),}
		\arrow["{(w, \pi_{E})}", from=1-1, to=3-1]
	\end{tikzcd}\] 
	which is $J_{f}$-covering since $w$ is an epimorphism. 
	
	This argument shows that every object of the category $(1_{\cal F}\downarrow f^{\ast})$ can be $J_{f}$-covered by an object lying in the full subcategory $(1_{\cal F}\downarrow^{\Sub} f^{\ast})$, as required.
\end{proof}

Let $L_{f}:{\cal E}^{\textup{op}}\to \CAT$ be the functor sending any object $E$ of $\cal E$ to the subobject lattice $\Sub_{\cal F}(f^{\ast}(E))$ and acting on arrows by pullback. The category $(1_{\cal F}\downarrow^{\Sub} f^{\ast})$  is clearly isomorphic to the category ${\cal G}(L_{f})$. 

Proposition \ref{prop:densenessrelsitelocalicmorphism} ensures that, by the Comparison Lemma, we have an equivalence of toposes
\[
\Sh((1_{\cal F}\downarrow f^{\ast}), J_{f})\simeq \Sh({\cal G}(L_{f}), J_{f}|_{{\cal G}(L_{f})}),
\] 
where $J_{f}|_{{\cal G}(L_{f})}$ is the Grothendieck topology induced by $J_{f}$ on the $J_f$-dense subcategory ${\cal G}(L_{f})\hookrightarrow (1_{\cal F}\downarrow f^{\ast})$. It is therefore natural to wonder whether this induced topology admits an intrinsic description directly in terms of the internal locale $L_{f}$. The answer to this question is affirmative, as we are going to see.

Let us first observe that we can express the property of a family of arrows 
\[
\{(f_i, e_i):(F_i, E_i, \alpha_i:F_i\mono f^\ast(E_i))\to (F, E, \alpha:F\mono f^{\ast}(E)) \mid i\in I\}
\]
in ${\cal G}(L_{f})$ to be $J_{f}$-covering in terms of the subobjects of $f^{\ast}(E)$ obtained by taking the images of the composite arrows $f^{\ast}(e_i)\circ \alpha_{i}$, as follows: $\{(f_i, e_i) \mid i\in I \}$ is $J_{f}$, that is, the family $\{f_i \mid i\in I\}$ is epimorphic, if and only if $\bigcup_{i\in I}\Im(f^{\ast}(e_i)\circ \alpha_{i})=1_{f^{\ast}(E)}$ in $\Sub_{\cal F}(f^{\ast}(E))$.

From a structural point of view, the subobjects $\Im(f^{\ast}(e_i)\circ \alpha_{i})$ are precisely the images of the subobjects $\alpha_i\in \Sub_{\cal F}(f^{\ast}(E_i))$ under the functor
\[
\exists_{f^{\ast}(e_i)}:\Sub_{\cal F}(f^{\ast}(E_i)) \to \Sub_{\cal F}(f^{\ast}(E)),
\]
which is left adjoint to the functor
\[
L_{f}(f^{\ast}(e_i))=(f^{\ast}(e_i))^{\ast}:L_{f}(E)=\Sub_{\cal F}(f^{\ast}(E))\to L_{f}(E_i)=\Sub_{\cal F}(f^{\ast}(E_i)). 
\]
Therefore the $J_f$-covering condition can be naturally expressed in terms of such left adjoints of transition morphisms of $L_{f}$.

Our analysis of localic geometric morphisms in terms of relative sites naturally brings us into the subject of internal locales. Indeed, as we shall see, $L_{f}$ is an internal locale to $\cal E$.

An internal locale to a Grothendieck topos is a model of the (higher-order) theory of frames in that topos. Johnstone characterized in \cite{elephant} the internal locales to a topos of sheaves on a cartesian site; specializing his characterization to the canonical site of a Grothendieck topos yields the following definition: 

\begin{defn}[cf. Lemma C1.6.9 \cite{elephant}]\label{definternallocale}
	Let $\cal E$ be a Grothendieck topos. An \emph{internal locale} to $\cal E$ is a functor $L:{\cal E}^{\textup{op}}\to \CAT$ satisfying the following conditions:
	
	\begin{enumerate}[(i)]
		\item For each $E\in {\cal E}$, $L(E)$ is a frame (regarded as a preorder category) and for each arrow $e:E\to E'$, $L(e):L(E')\to L(E)$ is a frame homomorphism admitting a left adjoint $\exists_{e}:L(E)\to L(E')$; 
		
		\item The underlying set-valued functor of $L$ is a sheaf for the canonical topology on $\cal E$; 
		
		\item $L$ satisfies the \emph{Beck-Chevalley condition}, that is, for any pullback square
		
		\[\begin{tikzcd}
			U & V \\
			W & Z
			\arrow["a", from=1-1, to=1-2]
			\arrow["c", from=1-2, to=2-2]
			\arrow["b"', from=1-1, to=2-1]
			\arrow["d", from=2-1, to=2-2]
			\arrow["\lrcorner"{anchor=center, pos=0.125}, draw=none, from=1-1, to=2-2]
		\end{tikzcd}\]
		
		in $\cal E$, the following square commutes:
		
		\[\begin{tikzcd}
			{L(V)} & {L(U)} \\
			{L(Z)} & {L(W)}
			\arrow["{L(a)}", from=1-1, to=1-2]
			\arrow["{\exists_b}", from=1-2, to=2-2]
			\arrow["{\exists_c}"', from=1-1, to=2-1]
			\arrow["{L(d)}", from=2-1, to=2-2]
		\end{tikzcd}\]
		
		\item $L$ satisfies the \emph{Frobenius reciprocity conditions}, that is, for any $a:E\to E'$, $l\in L(E)$ and $l'\in L(E')$,
		\[
		\exists_a(L(a)(l') \wedge l)= \exists_a(l) \wedge l'. 
		\] 	
	\end{enumerate}
\end{defn}

The above discussion shows that the description for the Grothendieck topology induced on ${\cal G}(L_{f})$ by $J_{f}$ can be entirely expressed in terms of the categorical structure present on $L_{f}$ (the transition morphisms and their left adjoints). In fact, such a description actually makes sense for an arbitrary internal locale $L:{\cal E}^{\textup{op}}\to \CAT$ in $\cal E$, yielding a relative site of presentation for the structure geometric morphism $\Sh_{\cal E}(L)\to {\cal E}$: the underlying category is ${\cal G}(L)$, and the relative topology, which we denote by $J_{L}^{\textup{ext}}$ (note that this topology generalizes the canonical topology on a frame $L$ in $\Set$), is defined as follows: a family
\[
\{e_i:(E_i, l_i)\to (E, l) \mid i\in I\}
\] 
of arrows in ${\cal G}(L)$ is $J_{L}^{\textup{ext}}$-covering if and only if, denoting by
\[
\exists_{e_i}:L(e_i)\to L(e)
\]
the left adjoints to the transition maps $L(e_i):L(e)\to L(e_i)$, $l=\bigvee \exists_{e_i}(l_i)$ in $L(e)$.

As we shall see in the proof of Theorem \ref{thmmainexistential}, the Beck-Chevalley condition and the Frobenius reciprocity law ensure that $J_{L}^{\textup{ext}}$ is a well-defined Grothendieck topology on ${\cal G}(L)$.

As it is well known, localic geometric morphisms towards a topos $\cal E$ can all be represented as toposes of sheaves on an internal locale to $\cal E$. In fact, there is an equivalence between internal locales in $\cal E$ and localic morphisms to $\cal E$:

\begin{prop}\label{propequivlocalicmorphisms}
	Let $\cal E$ be a Grothendieck topos. Then there is an equivalence
	\[
	\textup{Loc}({\cal E}) \simeq \mathfrak{Loc}\slash {\cal E} 
	\]
	where $\textup{Loc}({\cal E})$ is the category of internal locales in $\cal E$ and the slice category $\mathfrak{Loc}\slash {\cal E}$, where $\mathfrak{Loc}$ is the category of Grothendieck toposes and localic morphisms between them, given by the following correspondences:

Given an internal locale $L$ in $\cal E$, we associate with it the structure morphism $\Sh_{\cal E}(L)\to {\cal E}$, which can be represented as $C_{p_{L}}:\Sh({\cal G}(L), J_{L}^{\textup{ext}}) \to ({\cal E}, J^{\textup{can}}_{\cal E})$, where $p_{L}$ is the canonical projection functor.

Conversely, given a (localic) geometric morphism $f:{\cal F}\to {\cal E}$, we associate with it the functor $L_{f}:{\cal E}^{\textup{op}}\to \CAT$ sending each $E\in {\cal E}$ to the frame $\Sub_{\cal F}(f^{\ast}(E))$, which yields an internal locale in $\cal E$.
\end{prop}

\begin{proof}
	The morphism $C_{p_{L}}$ is localic by Proposition 7.11 \cite{denseness}, since $p_{L}:{\cal G}(L)\to {\cal E}$ is clearly faithful. 
	
	The fact that, for any morphism $f$, $L_{f}$ is an internal locale, that is, satisfies the conditions of Definition \ref{definternallocale}, follows at once from the categorical properties of the geometric morphism $f$.
	
	It remains to show that the two correspondences are quasi-inverse to each other. Given a localic morphism $f:{\cal F}\to {\cal E}$, the fact that $f\cong C_{p_{L_{f}}}$ over $\cal E$ follows from the above discussion. Given an internal locale $L$ in $\cal E$, we want to show that the canonical morphism 
	\[
	\eta_{L}:L \to L_{C_{p_{L}}}
	\] 
	of Proposition \ref{prop:unitfibredpreorder} is an isomorphism. But this follows from Proposition \ref{propsitecharinternallocales}.
	
	The functoriality of the construction of $L_{f}$ from $f$ is clear, while functoriality of the assignment $L \mapsto \Sh_{\cal E}(L)$, as well as its fullness and faithfulness, follows from Corollary \ref{cor:fullandfaithfulnessinternallocales}.	
\end{proof}

\begin{remarks}
	\begin{enumerate}[(a)]
		\item The content of Proposition \ref{propequivlocalicmorphisms} is essentially, except for the functorial part, that of Lemma 1.2 of \cite{JohnstoneFactI}. However, in that paper the construction of the two functors is presented in non-explicit terms (the localic morphism corresponding to an internal locale is described as \ac the topos of sheaves on the internal locale with respect to the canonical topology', and there is no indication of a functor between sites inducing this morphism; on the other hand, the internal locale corresponding to a localic morphism $f:{\cal F}\to {\cal E}$ is presented therein as $f_{\ast}(\Omega)$, where $\Omega$ is the subobject classifier in $\cal F$). In other words, our treatment differs from Johnstone's one in being fully explicit and written in the language of stacks and indexed categories rather than in that of internal categories, according to the foundations for relative topos theory introduced in \cite{CaramelloZanfa}. Our explicit, site-theoretic treatement notably allows us to obtain specific results about these concepts, of which a selection is provided below. 
		
		\item The category $\mathfrak{Loc}\slash {\cal E}$ coincides with the category having an objects the localic morphisms towards $\cal E$ and as arrows \emph{all} the morphisms between them. Indeed, as remarked in \cite{JohnstoneFactI} is that if $g\circ f$ is localic and $f$ is localic then $g$ is localic (this is actually a formal consequence of the fact that, as shown in \cite{JohnstoneFactI}, localic morphisms form the right class of an orthogoncal factorization system). 
		
		\item It follows from Proposition \ref{propequivlocalicmorphisms} that every the transition morphisms of an internal locale admit not only left adjoints, but also right adjoints; indeed, this is clearly the case for all the internal locales of the form $L_{f}$ for a geometric morphism $f$ and every internal locale is, up to isomorphism, of this form.
		
		\item The notion of internal locale, as specified by the above axioms, can be seen as an axiomatisation for the hyperdoctrines (in the sense of \cite{hyperdoctrines}) arising from (localic) geometric morphisms. We shall obtain a topos-theoretic analogue of this below (in the context of existential fibred sites).   
	\end{enumerate}
\end{remarks}

\begin{defn}\label{defreducedrelativesite}
	For a localic geometric morphism $f:{\cal F}\to {\cal E}$, we shall call the site $p_{L_{f}}:({\cal G}(L_{f}), J^{\textup{can}}_{L_{f}}) \to ({\cal E}, J^{\textup{can}}_{\cal E})$ over $\cal E$ the \emph{reduced relative site} of $f$. 
\end{defn}

We have seen that the inclusion $i_{f}:{\cal G}(L_{f}) \hookrightarrow (1_{\cal F}\downarrow f^{\ast})$ is $J_{f}$-dense. In fact, it is both a morphism and a comorphism of sites $i_{f}:({\cal G}(L_{f}), J^{\textup{can}}_{L_{f}}) \hookrightarrow ((1_{\cal F}\downarrow f^{\ast}), J_{f})$ inducing equivalences of toposes $\Sh(i_{f})$ $C_{i_{f}}$ which are quasi-invese to each other. In fact, $i_{f}$ is a morphism of fibrations over $\cal E$.

Still, the canonical relative site of $f$ and the reduced relative site of $f$ are also related by another functor, namely the functor $t_{f}:(1_{\cal F}\downarrow f^{\ast}) \to {\cal G}(L_{f})$ sending and object $(F, E, \alpha:F \to f^{\ast}(E))$ to the object $(\dom(\textup{Im}(\alpha)), E,  \textup{Im}(\alpha))$ of ${\cal G}(L_{f})$ and acting on the arrows accordingly. This functor is actually a morphism of fibrations as well as a comorphism of sites $((1_{\cal F}\downarrow f^{\ast}), J_{f}) \to ({\cal G}(L_{f}), J^{\textup{can}}_{L_{f}})$. Since $t_{f}\circ i_{f}=1_{{\cal G}(L_{f})}$, it follows, by functoriality of the construction of the geometric morphism induced by a comorphism of sites, that $C_{t_{f}}$ provides a quasi-inverse to $C_{i_{f}}$. 

Summarizing, we have the following result:

\begin{prop}
	Let $f:{\cal F}\to {\cal E}$ be a localic geometric morphism. Then $f$ can be represented in terms of relative sites as displayed in the following diagram:
	\[\begin{tikzcd}
		{\textup{\bf Sh}({\cal G}(L_{f}), J^{\textup{can}}_{L_{f}})} & {} & {\textup{\bf Sh}((1_{\cal F}\downarrow f^{\ast}), J_{f})} && {{\cal F}} \\
		\\
		&&& {{\cal E}}
		\arrow["f", from=1-5, to=3-4]
		\arrow["{C_{\pi_{\cal F}}}", from=1-3, to=1-5]
		\arrow["\sim"', from=1-3, to=1-5]
		\arrow["{C_{\pi_{\cal E}}}"', from=1-3, to=3-4]
		\arrow["{C_{i_{f}}}", shift left=1, from=1-1, to=1-3]
		\arrow["{C_{p_{L_{f}}}}"', from=1-1, to=3-4]
		\arrow["{C_{t_{f}}}", shift left=1, from=1-3, to=1-1]
	\end{tikzcd}\]
	
\end{prop}\qed

\section{Existential fibred sites}\label{sec:existentialfibredsites}

In this section we shall introduce a wide generalization (from the setting of fibred preorder sites to that of fibred sites) of the notion of internal locale and show that, as in the case of internal locales, one can define a Grothendieck topology on the Grothendieck's category of such an indexed category entirely in terms of the structure present on its fibers and the transition morphisms between them (as well as their left adjoints). 

This kind of relative sites is notably relevant in connection with \ac relative dualities', in an analogous way as sites arising by equipping ordinary categories with Grothendieck topologies defined entirely in terms of their categorical structure play an important role in connection with the topos-theoretic understanding of dualities (cf. \cite{caramello2011topostheoretic}). 

We will show that \emph{every} geometric morphism can be represented as the topos of sheaves on an existential fibred site. This representation is actually very natural since, by the relative Diaconescu's theorem established in section \ref{sec:relativeDiaconescu}, the morphisms between relative geometric morphisms correspond precisely to the morphisms between these existential fibred sites. 

This result can be seen as a generalization of the fact that the category of locales (resp. internal locales) embeds fully and faithfully in the category of toposes (resp. in the category of relative toposes), and also of the characterization of morphisms of toposes as morphisms between their canonical sites.

Our construction of the existential topology can also be used to associate toposes with theories formalized as hyperdoctrines (in the sense of \cite{hyperdoctrines}), as the latter give naturally rise to existential fibred sites (one may actually have different hyperdoctines depending on whether the given theory belongs to a smaller fragment of geometric logic such as regular or coherent logic, since the left adjoints exist already at the level of those fragments). In particular, in the context of our development of \ac relative geometric logic', it will be interesting to investigate to which extent geometric theories relative to a base topos $\cal E$ can be presented as existential fibred sites over the canonical site of $\cal E$. As it will be discussed in section \ref{sec:hyperconn}, considering ordinary geometric theories as internal locales to the classifying topos of simpler theories (such as the empty theory over their signature) yields, for instance, alternative sites of definition for their classifying toposes (cf. \cite{WrigleyToposesOnline} and the forthcoming paper by the same author).

\begin{defn}
Let $({\cal C}, J)$ be a small-generated site. 

\begin{enumerate}[(a)]
	\item A \emph{fibred site over ${\cal C}$} (in the sense of the \emph{Exposé VI} of \cite{SGA4_II}) is an indexed category $L:{\cal C}^{\textup{op}} \to \CAT$ taking values in the category of small-generated sites and morphisms of sites between them; we shall denote by $J^{L}_{c}$ the Grothendieck topology on the fiber $L(c)$.
	
	\item A \emph{fibred site over $({\cal C}, J)$} is a $J$-stack $L:{\cal C}^{\textup{op}} \to \CAT$ taking values in the category of small-generated sites and morphisms of sites between them.
	
	\item A \emph{fibred preorder site} is a fibred site which takes values in the category of small-generated preorder sites and morphisms of sites between them.	
	
	\item A fibred site $L:{\cal C}^{\textup{op}} \to \CAT$ over $\cal C$ is said to be \emph{$J$-reflecting}, where $J$ is a Grothendieck topology on $\cal C$, if for any $J$-covering family $S$ on an object $c$ of $\cal C$ and any family $T$ of arrows with common codomain in the category $L(c)$, if $L(f)(T)$ is $J^{L}_{\dom(f)}$-covering in the category $L(\dom(f))$ for every $f\in S$ then $T$ is $J^{L}_{c}$-covering (note that the converse implication always holds since by our hypotheses the transition functors are cover-preserving).
	
	\item A fibred site $L:{\cal C}^{\textup{op}} \to \CAT$ over ${\cal C}$ is said to be \emph{existential} if for any arrow $a:E' \to E$ in $\cal C$, the transition functor $L(a):L(E)\to L(E')$ has a left adjoint, denoted $\exists_{a}:L(E')\to L(E)$, and the following two conditions are satisfied ($\eta_{f}$ denotes the unit of the adjunction $\exists_f \dashv L(f)$, for each arrow $f$, and $\overline{(-)}$ denotes the operation of transposition along the adjunction):

	\begin{enumerate}[(i)]
		\item \emph{Relative Beck-Chevalley condition:}
	 
		For any arrows $c:V\to Z$ and $d:W\to Z$ in $\cal C$ with common codomain and any $l\in L(W)$, the family of arrows $$\{\overline{\widetilde{L(a)(\eta_{c}(l))}}:(\exists_{b})(L(a)(l)) \to L(d)(\exists_c(l)) \mid (a, b)\in  B_{(c, d)} \}$$ is $J^{L}_{V}$-covering, where $B_{(c, d)}$ is the collection of spans $(a:U\to V, b:U\to W)$ such that $c\circ a=d\circ c$
		\[\begin{tikzcd}
			U & V \\
			W & Z
			\arrow["c", from=1-2, to=2-2]
			\arrow["d", from=2-1, to=2-2]
			\arrow["a", from=1-1, to=1-2]
			\arrow["b"', from=1-1, to=2-1]
		\end{tikzcd}\]
		
		and $\overline{\widetilde{L(a)(\eta_{c}(l))}}$ is the transpose of the arrow
		\[
		\widetilde{L(a)(\eta_{c}(l))}: L(a)(l) \to L(b)(L(d)(\exists_c(l)))
		\]
		given by the composite of the arrow $L(a)(\eta_{c}(l))$ with the inverse of the isomorphism
		\[
		L(b)(L(d)(\exists_c(l))) \to L(a)(L(c)(\exists_c(l))) 
		\]
	    resulting from the equality $c\circ a=d\circ b$ in light of the pseudofonctoriality of $L$. 
		
		\item \emph{Relative Frobenius condition:} For any arrows $f:E\to E'$ in $\cal C$, any $l\in L(E)$ and any arrow $\alpha:l'\to \exists_{f}(l)$, the family of arrows $\{\overline{\delta}:\exists_{f}(m) \to l' \mid  (\delta, \rho)\in Q_{(f, l, \alpha)}\}$ is $J^{L}_{E'}$-covering, where $Q_{(f, l, \alpha)}$ is the collection of spans of arrows $(\rho:m\to l, \delta:m\to L(f)(l'))$ in $L(E)$ which make the rectangle
		\[\begin{tikzcd}
		m && l \\
		{L(f)(l')} && {L(f)(\exists_f(l))}
		\arrow["{\eta_{f}(l)}", from=1-3, to=2-3]
		\arrow["\rho", from=1-1, to=1-3]
		\arrow["\delta"', from=1-1, to=2-1]
		\arrow["{L(f)(\alpha)}"', from=2-1, to=2-3]
		\end{tikzcd}\]
		commute (that is, such that $\alpha\circ \overline{\delta}=\exists_{f}(\rho)$)
		 
		[Note that this is a form of denseness of the functor $\exists_f$ in the sense of \cite{denseness}, since it states that any arrow $\alpha$ to an object in the image of $\exists_f$ can be locally represented in terms of arrows in the image of $\exists_f$.]
	\end{enumerate} 

\item An existential fibred site is said to be \emph{open} if, for any arrow $a:E'\to E$ in $\cal C$, the functor $\exists_{f}:(L(E'), J^{L}_{E'})\to (L(E), J^{L}_{E})$ is cover-preserving. 

\item A \emph{morphism $\alpha:L \to L'$ of existential fibred sites} over $({\cal C}, J)$ is a morphism of fibrations $L \to L'$ over $\cal C$ such that for each $E\in {\cal C}$, $\alpha_{E}$ is a morphism of sites $(L(E), J^{L}_{E})\to (L'(E), J^{L'}_{E'})$ and for any arrow $a:E' \to E$ in $\cal C$ the following diagram commutes up to isomorphism:
\[\begin{tikzcd}
	{L(E')} & {L'(E')} \\
	{L(E)} & {L'(E)}
	\arrow["{\exists^{L}_{a}}"', from=1-1, to=2-1]
	\arrow["{\exists^{L'}_{a}}", from=1-2, to=2-2]
	\arrow["{\alpha_{E}}", from=1-1, to=1-2]
	\arrow["{\alpha_{E'}}", from=2-1, to=2-2]
\end{tikzcd}\]

\end{enumerate}
\end{defn}

\begin{thm}\label{thmmainexistential}
	Let $({\cal C}, J)$ be a small-generated site and $L:{\cal C}^{\textup{op}} \to \CAT$ an existential fibred site over ${\cal C}$. Then the families on the category ${\cal G}(L)$ of the form 
	\[
	\{(e_i, \alpha_i):(E_i, l_i)\to (E, l) \mid i\in I\}
	\] 
	 where the family $\{\overline{\alpha_i}:\exists_{e_i}(l_i) \to l \mid i\in I\}$ is $J^{L}_{L(e)}$-covering are the covering families for a Grothendieck topology $J_{L}^{\textup{ext}}$, called the \emph{existential topology}, on ${\cal G}(L)$.
	 
	 Conversely, if $L$ is a fibred site over $\cal C$ with left adjoints $\exists_f$ to the transition functors $L(f)$ for any arrow $f$ in $\cal C$ such that the families 	
	 \[
	 \{(e_i, \alpha_i):(E_i, l_i)\to (E, l) \mid i\in I\}
	 \] 
	 such that $\{\overline{\alpha_i}:\exists_{e_i}(l_i) \to  l \mid i\in I\}$ is $J^{L}_{L(e)}$-covering form a Grothendieck topology on ${\cal G}(L)$ then $L$ is existential. 
\end{thm}

\begin{proof}
	It is easy to see that the families in $J_{L}^{\textup{ext}}$ are stable under multicomposition and that any family through which a family in $J_{L}^{\textup{ext}}$ factors lies also in $J_{L}^{\textup{ext}}$. So, $J_{L}^{\textup{ext}}$ is a Grothendieck topology if and only if it satisfies the pullback stability property. 
	
	Since every arrow $\{(e, \alpha):(E', l')\to (E, l)\}$ in ${\cal G}(L)$ can be written as the composite of the arrow $(e, \eta_{e}(l'):(E', l') \to (E, \exists_{e}(l')))$ with the arrow $(1_{E}, \overline{\alpha}):(E, \exists_{e}(l')) \to (E, l)$, where $\overline{\alpha}$ is the transpose of the arrow $\alpha:l'\to L(e)(l)$ along the adjunction $(\exists_{e} \dashv L(e))$, it is sufficient (and necessary) to show, in light of the behaviour of the multicomposition operation for sieves with respect to the pullback operation, that the pullbacks of families in a fiber which are covering in the topology on the fiber are covering and that pullbacks of arrows of the form $(e, \eta_{e}(l'))$ are covering. Now, the fact that pullbacks of covering families in a fiber along an horizontal arrow in ${\cal G}(L)$ are covering follows from the fact that transition morphisms are cover-preserving, while the fact that the pullbacks of covering families in a fiber along a vertical arrow in ${\cal G}(L)$ are covering follows from the fact that the topology on the fiber satisfies the pullback-stability axiom. It thus remains to consider the other two cases, that is, the pullbacks of an arrow of the form $(e, \eta_{e}(l'))$ along horizontal arrows and the pullbacks of an arrow of the form $(e, \eta_{e}(l'))$ along vertical arrows. Let us show that   
	
	\begin{enumerate}[(1)]
		\item The relative Beck-Chevalley condition states precisely that for any arrows $c:V\to Z$ and $d:W\to Z$ in $\cal C$ and any $l\in L(V)$, the sieve $(d, 1)^{\ast}(<(c, \eta_{c}(l))>)$ is $J_{L}^{\textup{ext}}$-covering:
		
		\[\begin{tikzcd}
			& {(V, l)} \\
			{(W, L(d)(\exists_c(l)))} & {(Z, \exists_c(l))}
			\arrow["{(c, \exists_c(l))}", from=1-2, to=2-2]
			\arrow["{(d, 1)}", from=2-1, to=2-2]
		\end{tikzcd}\] 
		
		
		\item The relative Frobenius condition states precisely that for any arrow $f:E\to E'$ in $\cal C$ and any arrow $\alpha:l'\to \exists_{f}(l)$, the sieve $$(1_{E'}, \alpha)^{\ast}(<(f, \eta_{f}(l))>)$$ is $J_{L}^{\textup{ext}}$-covering:
		
		\[\begin{tikzcd}
			&& {(E, l)} \\
			{(E', l')} && {(E', \exists_f(l))}
			\arrow["{(f, \eta_{f}(l))}", from=1-3, to=2-3]
			\arrow["{(1_{E'}, \alpha)}", from=2-1, to=2-3]
		\end{tikzcd}\]
	\end{enumerate}  
	
	(1) The condition that for any arrows $c:V\to Z$ and $d:W\to Z$ in $\cal C$ and any $l\in L(V)$, the sieve $(d, 1)^{\ast}(<(c, \eta_{c}(l))>)$ be $J_{L}^{\textup{ext}}$-covering is equivalent to requiring that the family of arrows $$\{\overline{\chi}:(\exists_{b})(l') \to L(d)(\exists_c(l)) \mid (\chi, \xi, l', a, b)\in  B_{(c, d, l)} \}$$ be $J^{L(V)}$-covering, where $B_{(c, d, l)}$ is the collection of quintuples $(a:U\to V, b:U\to W, l'\in L(U), \chi:l'\to L(b)(L(d)(\exists_c(l)))\cong L(a)(L(c)(\exists_c(l))), \xi:l'\to L(a)(l))$ such that the following diagrams commute:
	 
	 \[\begin{tikzcd}
	 	U & V && {l'} & {L(a)(l)} \\
	 	W & Z && {L(b)(L(d)(\exists_c(l)))} & {L(a)(L(c)(\exists_c(l)))}
	 	\arrow["a", from=1-1, to=1-2]
	 	\arrow["b"', from=1-1, to=2-1]
	 	\arrow["d"', from=2-1, to=2-2]
	 	\arrow["c", from=1-2, to=2-2]
	 	\arrow["\chi"', from=1-4, to=2-4]
	 	\arrow["\xi", from=1-4, to=1-5]
	 	\arrow["\cong", from=2-4, to=2-5]
	 	\arrow["{L(a)(\eta_{c}(l))}", from=1-5, to=2-5]
	 \end{tikzcd}\]
(where the isomorphism in the square on the right-hand side results from the equality $c\circ a=d\circ b$ in light of the pseudofonctoriality of $L$); indeed, the arrows $(b, \chi)$ belonging to this sieve are parametrized by the quintuples in $B_{(c, d, l)}$:
	\[\begin{tikzcd}
		{(U, l')} & {(V, l)} \\
		{(W, L(d)(\exists_c(l)))} & {(Z, \exists_c(l))}
		\arrow["{(c, \exists_c(l))}", from=1-2, to=2-2]
		\arrow["{(d, 1)}", from=2-1, to=2-2]
		\arrow["{(a, \xi)}", from=1-1, to=1-2]
		\arrow["{(b, \chi)}", from=1-1, to=2-1]
	\end{tikzcd}\]
But every arrow $\overline{\chi}$ as above factors through the arrow
	\[
	\widetilde{L(a)(\eta_{c}(l))}: L(a)(l) \to L(b)(L(d)(\exists_c(l)))
	\]
	given by the composite of the arrow $L(a)(\eta_{c}(l))$ with the inverse of the isomorphism
	\[
	L(b)(L(d)(\exists_c(l))) \to L(a)(L(c)(\exists_c(l))) 
	\]
	appearing in the above square, so the condition that the family of such arrows be covering amounts precisely to the condition that such arrows yield a covering family, that is to the condition in the statement of the relative Beck-Chevalley condition.
	
	(2) The condition that the family of arrows $\{\overline{\delta}:\exists_{f}(m) \to l' \mid  (\delta, \rho)\in Q_{(f, \alpha)}\}$ be $J^{L(E')}$-covering is clearly equivalent to the requirement that
	the family
	\[
	\{(f, \delta):(E, m) \to (E', l') \mid (\delta, \rho)\in Q_{(f, \alpha)}\},
	\] 
	be $J_{L}^{\textup{ext}}$-covering, and the commutativity of the square appearing in the condition amounts precisely to that of the following diagram:
	\[\begin{tikzcd}
		{(E, m)} && {(E, l)} \\
		{(E', l')} && {(E', \exists_f(l))}
		\arrow["{(f, \eta_f(l))}", from=1-3, to=2-3]
		\arrow["{(f, \delta)}", from=1-1, to=2-1]
		\arrow["{(1_E, \rho)}", from=1-1, to=1-3]
		\arrow["{(1_E', \alpha)}", from=2-1, to=2-3]
	\end{tikzcd}\]
	
	On the other hand, as shown by the following diagram, an arbitrary commutative square always factors through one of the above form:
	\[\begin{tikzcd}
		{(E'', l'')} && {(E, l)} \\
		& {(E, \exists_a(l''))} \\
		{(E', l')} && {(E', \exists_f(l))}
		\arrow["{(f, \eta_f(l))}", from=1-3, to=3-3]
		\arrow["{(b, \chi)}", from=1-1, to=3-1]
		\arrow["{(a, \xi)}"{pos=0.6}, from=1-1, to=1-3]
		\arrow["{(1_E', \alpha)}"{pos=0.7}, from=3-1, to=3-3]
		\arrow["{(a, \eta_{a}(l''))}"{pos=0.9}, from=1-1, to=2-2]
		\arrow["{(1_E, \overline{\xi})}"'{pos=0.7}, from=2-2, to=1-3]
		\arrow["{(f, \overline{\chi})}"{pos=0.1}, from=2-2, to=3-1]
	\end{tikzcd}\]
	
	Therefore, it suffices to consider sieves of the above form, as any family through which a covering family factors is also covering.
	
	From conditions (1) and (2) our thesis obviously follows.
\end{proof}

The above result justifies the following definition:

\begin{defn}
	Let $({\cal C}, J)$ be a small-generated site and $L:{\cal C}^{\textup{op}} \to \CAT$ an existential fibred site over ${\cal C}$. The relative topos 
	\[
	C_{\pi_{L}}:\Sh({\cal G}(L), J_{L}^{\textup{ext}}) \to \Sh({\cal C}, J)
	\] 
	is called the \emph{existential topos of $L$}.
\end{defn}

\begin{remarks}	
	\begin{enumerate}[(a)]
		\item The existential topology is generated by both \ac vertical' components (provided by the Grothendieck topologies on the fibers) and by \ac horizontal' components (provided by the left adjoints $\exists_e$); in fact, it strictly contains the total topology of the fibered site given by $L$. 	
		
		\item Any morphism of existential fibred sites $L\to L'$ over a small-generated site $({\cal C}, J)$ induces a morphism of sites $({\cal G}(L), J_{L}^{\textup{ext}}) \to ({\cal G}(L'), J_{L'}^{\textup{ext}})$ over $({\cal C}, J)$ and hence a morphism $[C_{\pi_{L'}}]\to [C_{\pi_{L}}]$ of existential toposes over $\Sh({\cal C}, J)$. 
		
		\item Given relative toposes $[f :{\cal F} \to {\cal E}]$ and $[f':{\cal F}\to {\cal E}]$, the
		geometric morphisms $f\to f'$ over $\cal E$ correspond precisely to the
		morphisms of existential fibred sites $E_{f'}\to E_{f}$ (cf. Corollary \ref{correlativemorphisms}).
	\end{enumerate}	
\end{remarks}

In the case of pseudofunctors $L$ defined on a cartesian category $\cal C$ and which the property that each category $L(c)$ has finite limits which are preserved by the transition functors, the relative Beck-Chevalley and the relative Frobenius condition admit simpler reformulations, as shown by the following proposition.

For this, we need a lemma:

\begin{lemma}
	Let $\cal C$ be a cartesian category and $L:{\cal C}^{\textup{op}}\to \CAT$ a $\cal C$-indexed category such that for any $c\in {\cal C}$, $L(c)$ is a cartesian category and for any $f:c\to c'$, $L(f):L(c')\to L(c)$ preserves finite limits. Then the category ${\cal G}(L)$ is cartesian; more specifically, $(1, 1_{L(1)})$ is the terminal object of ${\cal G}(L)$ and for any arrows $(f, \alpha):(c', x')\to (c, x)$ and $(g, \beta):(c'', x'')\to (c, x)$ in ${\cal G}(L)$, the square 
\[\begin{tikzcd}
	{(c'\times_c c'', z)} && {(c'', x'')} \\
	{(c', x')} && {(c, x)}
	\arrow["{(g, \beta)}", from=1-3, to=2-3]
	\arrow["{(f, \alpha)}", from=2-1, to=2-3]
	\arrow["{(\pi'', u'')}", from=1-1, to=1-3]
	\arrow["{(\pi', u')}", from=1-1, to=2-1]
\end{tikzcd}\]
is a pullback in ${\cal G}(L)$, where 
\[\begin{tikzcd}
	{c'\times_c c''} & {c''} \\
	{c'} & c
	\arrow["g", from=1-2, to=2-2]
	\arrow["f"', from=2-1, to=2-2]
	\arrow["{\pi'}"', from=1-1, to=2-1]
	\arrow["{\pi''}", from=1-1, to=1-2]
	\arrow["\lrcorner"{anchor=center, pos=0.125}, draw=none, from=1-1, to=2-2]
\end{tikzcd}\]
	and 
\[\begin{tikzcd}
	z && {L(\pi'')(x'')} \\
	{L(\pi')(x')} && {L(\pi')(L(f)(x))\cong L(\pi'')(L(g)(x))}
	\arrow["{L(\pi'')(\beta)}", from=1-3, to=2-3]
	\arrow["{L(\pi')(\alpha)}", from=2-1, to=2-3]
	\arrow["{u''}", from=1-1, to=1-3]
	\arrow["{u'}", from=1-1, to=2-1]
\end{tikzcd}\]
	are pullback squares respectively in $\cal C$ and in the category $L(c'\times_c c'')$. The element $z$ will be denoted $x'\times_c x''$. 
\end{lemma}

\begin{proof}
	Straightforward and left to the reader (cf. also Proposition 2.4.2 \cite{CaramelloZanfa}).
\end{proof}

\begin{prop}\label{propcartesianexistential}
Let $\cal C$ be a cartesian category and $L:{\cal C}^{\textup{op}}\to \CAT$ a $\cal C$-indexed category such that for any $c\in {\cal C}$, $L(c)$ is a cartesian category and for any $f:c\to c'$, $L(f):L(c')\to L(c)$ preserves finite limits. Then 

\begin{enumerate}[(i)]
	\item $L$ satisfies the relative Beck-Chevalley condition if and only if for any arrows $c:V\to Z$ and $d:W\to Z$ in $\cal C$ with common codomain and any $l\in L(W)$, the arrow $$\overline{u'}:\exists_{\pi'}( L(d)(\exists_c(l)) \times_Z l ) \to L(d)(\exists_c(l)) $$ is $J^{L}_{V}$-covering, where the following is a pullback square in ${\cal G}(L)$:
	\[\begin{tikzcd}
		{(V\times_Z W, L(d)(\exists_c(l))\times_Z l)} & {(W, l)} \\
		{(V, L(d)(\exists_c(l)))} & {(Z, \exists_c(l))}
		\arrow["{(c, \exists_c(l))}", from=1-2, to=2-2]
		\arrow["{(d, 1)}", from=2-1, to=2-2]
		\arrow[from=1-1, to=1-2]
		\arrow["{(\pi', u')}", from=1-1, to=2-1]
	\end{tikzcd}\]

	\item $L$ satisfies the relative Frobenius condition if and only if for any arrows $f:E\to E'$ in $\cal E$, any $l\in L(E)$ and any arrow $\alpha:l'\to \exists_{f}(l)$, the arrow
	$\overline{u'}:\exists_f(l'\times_{E'} l) \to l'$ is $J^{L}_{E'}$-covering, where the following is a pullback square in ${\cal G}(L)$:
	
	\[\begin{tikzcd}
		{(E, l'\times_{E'} l)} && {(E, l)} \\
		{(E', l')} && {(E', \exists_f(l))}
		\arrow["{(f, \eta_f(l))}", from=1-3, to=2-3]
		\arrow["{(f, u')}", from=1-1, to=2-1]
		\arrow["{(1_E, u'')}", from=1-1, to=1-3]
		\arrow["{(1_E', \alpha)}", from=2-1, to=2-3]
	\end{tikzcd}\]
\end{enumerate}
\end{prop} 

\begin{proof}
The proof follows immediately from the lemma in light of the characterization of the relative Beck-Chevalley and the relative Frobenius conditions established in the proof of Theorem \ref{thmmainexistential}.	
\end{proof}

\begin{defn}
	Let $\cal C$ be a category.	
	\begin{enumerate}[(a)]
		\item A family $S$ in $\cal C$ of arrows towards a given object $c$ is said to be \emph{linearizable} if there is a way of associating to $S$ an arrow $i_{S}:\dom(i_{S})\to c$ such that the every arrow in $S$ factors through $i_{S}$. We call the arrow $i_{S}$ a \emph{linearization} of $S$.
		
		\item A Grothendieck topology $J$ on $\cal C$ is said to admit a \emph{reflecting linearization} if every family $S$ in $\cal C$ of arrows with common codomain admits a linearization $i_{S}$ such that $S$ is $J$-covering if and only if $i_{S}$ is an isomorphism.
		
		\item A fibred site $L$ over $\cal C$ is said to admit a \emph{reflecting linearization} if all the Grothendieck topologies on its fibers admit a reflecting linearisation and the transition functors preserve such linearisations up to isomorphism (in the sense that for any $S$ and any $L(f)$, $i_{L(f)(S)}\cong L(f)(i_S)$ (as objects of the relevant slice category)). 
	\end{enumerate}
\end{defn}

\begin{ex}
	Most of the Grothendieck topologies considered in \cite{CaramelloStone}, which are specified by declaring a certain kind of joins which exist in a preorder covering, admit a reflecting linearization. In particular, the canonical topology on a topos and the canonical topology on a frame admit a reflecting linearization. Note that if a Grothendieck topology $J$ satisfies the property that the $J$-closed sieves are principal (i.e., generated by a single arrow) then it admits a reflecting linearisation, since any sieve $S$ is $J$-covering if and only if its $J$-closure is. Indeed, this is the situation for the above-mentioned Grothendieck topologies. 
\end{ex}

\begin{prop}\label{propcontainmentGiraudtopology}
	Let $L:{\cal C}^{\textup{op}}\to \CAT$ be a $J$-reflecting existential fibred site and $J$ a Grothendieck topology on $\cal C$. Then the existential topology $J^{\textup{ext}}_{L}$ contains the Giraud topology induced by $J$.
\end{prop}	

\begin{proof}
	Let $S$ be a $J$-covering family on an object $c$ of $\cal C$. We want to show that, for any $l\in L(c)$, the family $\epsilon^{f}(l):\exists_f(L(f)(l))\to l$, where $\epsilon^f$ is the counit of the adjunction $(\exists_f \dashv L(f))$ is $J^{L}_{c}$-covering. Since $L$ is $J$-reflecting, it is equivalent to show that for any $f\in S$ the image under $L(f)$ of this family is $J^{L}_{L(\dom(f))}$-covering. But this is the case, since by one of the triangular identities, such a family contains the identity on $L(\dom(f))$.
\end{proof}

\begin{prop}\label{propreflectingfibredsite}
	Let $L:{\cal C}^{\textup{op}}\to \CAT$ be a fibred site. If $L$ is a $J$-prestack and it admits a reflecting linearization then $L$ is $J$-reflecting.
\end{prop}

\begin{proof}
Given a family $S$ of arrows with common codomain $c$ in $\cal C$, let us consider a linearisation $i_{S}$. Let us suppose that for a $J$-covering family $T$ on $c$, $L(g)(S)$ is $J^{L}_{\dom(g)}$-covering. We want to show that $S$ is $J^{L}_{c}$-covering. Since $L$ admits a reflective linearisation, $S$ is $J^{L}_{c}$-covering if and only if $i_{S}$ is an isomorphism, and  $L(g)(S)$ is $J^{L}_{\dom(g)}$-covering if and only if $i_{L(g)(S)}=L(g)(i_{S})$ is an isomorphism. 

But, since $L$ is a $J$-prestack, $i_{S}$ is an isomorphism if and only if $L(g)(i_{S})$ is for any $g\in T$, whence our thesis follows. 	
\end{proof}	

\begin{defn}
	Let $f:{\cal F}\to {\cal E}$ be a geometric morphism. The \emph{existential fibred site of $f$} is the indexed functor $E_{f}:{\cal E}^{\textup{op}}\to \CAT$ sending any object $E$ of $\cal E$ to the topos ${\cal F}\slash f^{\ast}(E)$ endowed with its canonical topology (for any arrow $k:E'\to E$ in $\cal E$, the pullback functor $E_{f}(k):=(f^{\ast}(k))^{\ast}:{\cal F}\slash f^{\ast}(E)\to {\cal F}\slash f^{\ast}(E')$ has a left adjoint $\exists_{k}:{\cal F}\slash f^{\ast}(E')\to {\cal F}\slash f^{\ast}(E)$ given by composition with $f^{\ast}(k)$).  
	
	If $({\cal C}, J)$ is a site of definition for $\cal E$, the composite of $E_{f}$ with the canonical functor ${\cal C}\to \Sh({\cal C}, J)$ is also called the existential fibred site of $f$. 
\end{defn}

\begin{remark}
	For any geometric morphism $f$, the fibred site $E_{f}$ is open and has a reflective linearization.
\end{remark}

\begin{prop}
	Let $f:{\cal F}\to {\cal E}$ be a geometric morphism. Then, under the identification $(1 \downarrow f^{\ast})\cong {\cal G}(E_{f})$, the topology $J_{f}$ on $(1 \downarrow  f^{\ast})$ corresponds to the existential topology $J^{\textup{ext}}_{E_{f}}$ on ${\cal G}(E_{f})$ associated with the existential fibred site of $f$.  
\end{prop}

\begin{proof}
	Obvious from the discussion in section \ref{sec:localicmorphisms}.
\end{proof}

The following result shows that, for any existential fibred site $L$, the toposes of sheaves on the fibers of $L$ are related through hyperconnected morphisms to the fibers of the existential topos of $L$:

\begin{prop}\label{propfibers}
	Let $({\cal C}, J)$ be a small-generated site and $L$ an existential fibred site over $({\cal C}, J)$ and $c$ an object of $\cal C$. Let ${\cal G}^{\textup{ext}}_{c}(L)$ be the category whose objects are the pairs $(f, y)$ where $f$ is an arrow in $\cal C$ with codomain $c$ and $y\in L(\dom(f))$, and whose arrows $(f, y) \to (f', y')$ are the pairs $(\xi, \delta)$ consisting of an arrow $\xi$ in $\cal C$ such that $f'\circ \xi=f$ and an arrow $\delta:y\to L(\xi)(y')$ in the category $L(\dom(f))$. Let $\tilde{J}_{c}$ be the Grothendieck topology defined by saying that a family 
	\[
	\{(\xi_i, \delta_i): (f_i, y_i) \to (f, y) \mid  i\in I \}
	\]
	is $\tilde{J}_{c}$-covering if and only if the family 
	\[
	\{ \overline{\delta_i}:(\exists_{f_i}(y_i)) \to y \mid i\in I\}
	\]
	is $J^{L}_{\dom(f)}$-covering.
	
	Then:
	
	\begin{enumerate}[(i)]
		\item Then the canonical functor $p^{L}_{c}:{\cal G}^{\textup{ext}}_{c}(L) \to {\cal G}(L)$ is a comorphism of sites $({\cal G}^{\textup{ext}}_{c}(L), \tilde{J}_{c}) \to ({\cal G}(L), J_{L}^{\textup{ext}})$.
		
		\item The fibre $\Sh({\cal G}(L), J_{L}^{\textup{ext}}) \slash C_{\pi_{L}}^{\ast}(l(c))$ at $c$ of the existential topos $$C_{\pi_{L}}:\Sh({\cal G}(L), J_{L}^{\textup{ext}}) \to \Sh({\cal C}, J)$$ of $L$ is equivalent over  $\Sh({\cal G}(L), J_{L}^{\textup{ext}})$ to the topos $\Sh({\cal G}^{\textup{ext}}_{c}(L), \tilde{J}_{c})$:
		
		\[\begin{tikzcd}
			{\textup{\bf Sh}({\cal G}(L), J_{L}^{\textup{ext}}) / C_{\pi_{L}}^{\ast}(l(c))} && {\textup{\bf Sh}({\cal G}^{\textup{ext}}_{c}(L), \tilde{J}_{c})} \\
			& {\textup{\bf Sh}({\cal G}(L), J_{L}^{\textup{ext}})}
			\arrow[from=1-1, to=2-2]
			\arrow["{C_{p^L_{c}}}"{pos=0.4}, from=1-3, to=2-2]
			\arrow["\sim", from=1-1, to=1-3]
		\end{tikzcd}\]
		
		\item For any arrow $k:c\to c'$ in $\cal C$, the pullback functor
		\[
		\Sh({\cal G}(L), J_{L}^{\textup{ext}}) \slash C_{\pi_{L}}^{\ast}(l(c')) \to \Sh({\cal G}(L), J_{L}^{\textup{ext}}) \slash C_{\pi_{L}}^{\ast}(l(c))
		\]
		along the arrow $(C_{\pi_{L}})^{\ast}(l(k)):C_{\pi_{L}}^{\ast}(l(c)) \to C_{\pi_{L}}^{\ast}(l(c'))$ admits a left adjoint, given by the composition functor $\Sigma_{(C_{\pi_{L}})^{\ast}(l(k))}$ with $(C_{\pi_{L}})^{\ast}(l(k))$, which is induced by the comorphism of sites 
		\[
		({\cal G}^{\textup{ext}}_{c}(L), \tilde{J}_{c}) \to ({\cal G}^{\textup{ext}}_{c'}(L), \tilde{J}_{c'})
		\]
		given by composition with $k$.
		
		\item The morphism of sites
		\[
		i_{c}:(L(c), J^{L}_{c}) \to ({\cal G}^{\textup{ext}}_{c}(L), \tilde{J}_{c})
		\]
		sending an object $x$ of $L(c)$ to the object $(1_{c}, x)$ of ${\cal G}^{\textup{ext}}_{c}(L)$, and the (left adjoint) comorphism of sites
		\[
		\textup{ext}_{c}:({\cal G}^{\textup{ext}}_{c}(L), \tilde{J}_{c}) \to (L(c), J^{L}_{c})
		\]
		sending an object $(f, y)$ of ${\cal G}^{\textup{ext}}_{c}(L)$ to the object $\exists_{f}(y)$ of $L(c)$ induce 	an hyperconnected geometric morphism
		\[
		\Sh(i_{c})\cong C_{\textup{ext}_{c}}:\Sh({\cal G}^{\textup{ext}}_{c}(L), \tilde{J}_{c}) \to \Sh(L(c), J^{L}_{c}).
		\]
		
		\item For any arrow $k:c\to c'$ in $\cal C$, the following diagram of comorphism of sites commutes:
		\[\begin{tikzcd}
			{({\cal G}^{\textup{ext}}_{c}(L), \tilde{J}_{c})} & {(L(c), J^{L}_{c})} \\
			{({\cal G}^{\textup{ext}}_{c'}(L), \tilde{J}_{c'})} & {(L(c'), J^{L}_{c'})}
			\arrow["{E_k}", from=1-1, to=2-1]
			\arrow["{\exists_k}", from=1-2, to=2-2]
			\arrow["{\textup{ext}_c}", from=1-1, to=1-2]
			\arrow["{\textup{ext}_c'}", from=2-1, to=2-2]
		\end{tikzcd}\]
		
	\end{enumerate}
\end{prop}

\begin{proof}
	As a first remark, we observe that the category ${\cal G}^{\textup{ext}}_{c}(L)$ defined in the statement of the Proposition is precisely the category of elements of the functor $\Hom_{\cal C}(\pi_{L}(-), c):{\cal G}(L)^{\textup{op}}\to \Set$, and $\tilde{J}_{c}$ is the Grothendieck topology on it induced by the existential topology $J_{L}^{\textup{ext}}$ via the canonical projection functor ${\cal G}^{\textup{ext}}_{c}(L) \to {\cal G}(L)$.
	
	(i) This is a particular case of the fact that for any presheaf $Q$ on a category ${\cal D}$, endowed with a Grothendieck topology $K$, the category $\int Q$, endowed with the Grothendieck topology $K_{Q}$ induced by the canonical projection $\pi_{Q}:{\int Q} \to {\cal D}$, is a comorphism of sites $({\int Q}, K_{Q})\to ({\cal D}, K)$.
	
	(ii) This follows from the equivalence discussed in section 6.7 of \cite{denseness}.
	
	(iii) This follows from Theorem 3.3 of \cite{dependent}.
	
	(iv) First, we notice that $\textup{ext}_{c}$ is a comorphism of sites $({\cal G}^{\textup{ext}}_{c}(L), \tilde{J}_{c}) \to (L(c), J^{L}_{c})$. Indeed, this follows immediately from the Relative Frobenius condition. Next, we observe that, as $\textup{ext}_{c}$ is left adjoint to $i_{c}$ (as a consequence of the fact that, for any arrow $f$ (with codomain $c$), $\exists_{f}$ is left adjoint to $L(f)$), $i_{c}$ is a morphism of sites $(L(c), J^{L}_{}c)) \to ({\cal G}^{\textup{ext}}_{c}(L), \tilde{J}_{c})$ inducing the same geometric morphism (cf. Proposition 3.14 \cite{denseness}). 

	We can thus apply Proposition 6.25 \cite{denseness}, providing necessary and sufficient conditions for the geometric morphism induced by a morphism of sites to be hyperconnected: the geometric morphism $\Sh(F):\Sh({\cal D}, K) \to \Sh({\cal C}, J)$ induced by a morphism of small-generated sites $F:({\cal C}, J)\to ({\cal D}, K)$ is hyperconnected if and only if $F$ is cover-reflecting and \emph{closed-sieve-lifting}, in the sense that for every object $c$ of $\cal C$ and any $K$-closed sieve $S$ on $F(c)$ there exists a ($J$-closed) sieve $R$ on $c$ such that $S$ coincides with the $K$-closure of the sieve on $F(c)$ generated by the arrows $F(f)$ for $f\in R$.  
	
	The fact that $i_{c}:(L(c), J^{L}_{c})) \to ({\cal G}^{\textup{ext}}_{c}(L), \tilde{J}_{c})$ is cover-reflecting is clear by definition of the topology $\tilde{J}_{c}$. It remains to show that it is closed-sieve-lifting. For this, let us suppose that $S=\{(f, \delta): (f, y) \to (1_{c}, x) \mid  (f, \delta) \in S\}$ is a $\tilde{J}_{c}$-closed sieve on an object of the form $i_{c}(c)(x)=(1_{c}, x)$, for any $(f, \delta) \in S$, the arrow $(1_{c}, \overline{\delta}):(1_{c}, \exists_{f}(y)) \to (1_{c}, x)$ belongs to the $\tilde{J}_{c}$-closure of $S$, since the arrow $(f, \eta_{f}(y)):(f, y)\to (1_{c}, \exists_{f}(y))$ is $\tilde{J}_{c}$-covering and $(1_{c}, \overline{\delta})\circ (f, \eta_{f}(y))=(f, \delta)\in S$. So, since $S$ is $\tilde{J}_{c}$-closed by our hypothesis, for any $(f, \delta) \in S$, $(1_{c}, \overline{\delta}) \in S$, whence $S$ coincides with the $\tilde{J}_{c}$-closure of the image under the functor $i_{c}$ of the $J^L_{c}$-covering family $\{\overline{\delta}: \exists_{f}(y) \to x \mid (f, \delta)\in S \}$.
	
	(v) The commutativity of the diagram readily follows from the functoriality in $g$ of the construction $g\mapsto \exists_{g}$. 
\end{proof}

\begin{remark}\label{remretract}
	Under the hypotheses of Proposition \ref{propfibers}, for any $c\in {\cal C}$, $\textup{ext}_{c}\circ i_{c}\cong 1_{L(c)}$, whence the category $L(c)$ is a retract of the category ${\cal G}^{\textup{ext}}_{c}(L)$.  
\end{remark}

\subsection{Internal locales as existential fibred sites}

The following theorem, characterizing internal locales in an arbitrary presheaf topos, is an immediate corollary of Proposition \ref{propcartesianexistential}.

\begin{thm}\label{thm:existentialcartesian}
	Let $\cal C$ be a cartesian category and $L:{\cal C}^{\textup{op}}\to \CAT$ a $\cal C$-indexed category such that for any $c\in {\cal C}$, $L(c)$ is a frame, endowed with the canonical topology, and for any $f:c\to c'$, $L(f):L(c')\to L(c)$ is a frame homomorphism having a left adjoint $\exists_f$. Then $L$ is an existential fibred site if and only if it is an internal locale in $[{\cal C}^{\textup{op}}, \Set]$, that is, if and only if the following conditions are satisfied:
	
	\begin{enumerate}[(i)]
		\item \emph{Beck-Chevalley condition}: for any pullback square
		\[\begin{tikzcd}
			U & V \\
			W & Z
			\arrow["a", from=1-1, to=1-2]
			\arrow["c", from=1-2, to=2-2]
			\arrow["b"', from=1-1, to=2-1]
			\arrow["d", from=2-1, to=2-2]
			\arrow["\lrcorner"{anchor=center, pos=0.125}, draw=none, from=1-1, to=2-2]
		\end{tikzcd}\]
		
		in $\cal E$, the following square commutes:
		
		\[\begin{tikzcd}
			{L(V)} & {L(U)} \\
			{L(Z)} & {L(W)}
			\arrow["{L(a)}", from=1-1, to=1-2]
			\arrow["{\exists_b}", from=1-2, to=2-2]
			\arrow["{\exists_c}"', from=1-1, to=2-1]
			\arrow["{L(d)}", from=2-1, to=2-2]
		\end{tikzcd}\]
		
		\item \emph{Frobenius reciprocity condition}: for any $a:E\to E'$, $l\in L(E)$ and $l'\in L(E')$,
		\[
		\exists_a(L(a)(l') \wedge l)= \exists_a(l) \wedge l'. 
		\] 	
	\end{enumerate}
\end{thm}

\begin{proof}
	The Beck-Chevalley condition and the Frobenius condition for an internal locale $L$ respectively imply the relative Beck-Chevalley condition and the relative Frobenius condition. Indeed, the first assertion is obvious, while the second can be proved as follows. If $l'\leq \exists_f(l)$ then $l'=\exists_f(L(f)(l')\wedge l)$; indeed, by the Frobenius reciprocity condition, $\exists_f(L(f)(l')\wedge l)=l'\circ \exists_{f}(l)=l'$. So, by taking $m=l\wedge L(f)(l')$, $\rho$ to be the arrow $m=l\wedge L(f)(l') \leq l$ and $\delta$ to be the arrow $m=L(f)(l') \leq l \leq L(f)(l')$, the condition is trivially satisfied (independently from the Grothendieck topology $J^{L}_{E'}$ on the fiber $L(E')$, as the arrow $\overline{\delta}$ is the identify on $l'$).  
	
	Conversely, let us suppose that $L$ satisfies the relative Beck-Chevalley condition, and show that $L$ satisfies the (absolute) Beck-Chevalley condition. From Proposition \ref{propcartesianexistential} we know that, for any $l\in L(W)$, we have $L(d)(\exists_c(l))=\exists_b(L(d)(\exists_c(l))\times_{Z} l)$. But $L(d)(\exists_c(l))\times_{Z} l=L(a)(l)\wedge L(b)(L(d)(\exists_c(l)))=L(a)(l)\wedge L(a)(L(c)(\exists_{c}(l)))=L(a)(l)$ (since $l\leq L(c)(\exists_{c}(l))$).
	
	Lastly, let us suppose that that $L$ satisfies the relative Frobenius condition, and show that $L$ satisfies the Frobenius reciprocity condition. But this follows immediately by applying Proposition \ref{propcartesianexistential} in the particular case of $f$ being $a$ and $\alpha$ being the arrow $\exists_a(l)\wedge l' \to \exists_a(l)$.
\end{proof}

The following proposition gives a characterization of internal locales in a topos $\Sh({\cal C}, J)$ of sheaves on an arbitrary, not necessarily cartesian, small-generated site $({\cal C}, J)$:

\begin{prop}\label{propsitecharinternallocales}
	Let $({\cal C}, J)$ be a small-generated site. Then an internal locale in $\Sh({\cal C}, J)$ is a functor $L:{\cal C}^{\textup{op}}\to \CAT$ taking values in the subcategory of frames and frame homomorphisms which is a $J$-sheaf and, when considered as a fibred site (by endowing each frame with its canonical topology), is existential i.e. the following conditions are satisfied:
	
	\begin{enumerate}[(i)]
		\item \emph{Relative Beck-Chevalley condition:}
		For any arrows $c:V\to Z$ and $d:W\to Z$ in $\cal C$ with common codomain and any $l\in L(V)$,
		\[
		L(d)(\exists_c(l))=\mathbin{\mathop{\textup{\huge
					$\vee$}}\limits_{(a, b)\in B_{(c, d)}}}(\exists_{b}(L(a)(l))),
		\] 
		where $B_{(c, d)}$ is the collection of spans $(a:U\to V, b:U\to W)$ such that $c\circ a=d\circ c$: 
		\[\begin{tikzcd}
			U & V \\
			W & Z
			\arrow["c", from=1-2, to=2-2]
			\arrow["d", from=2-1, to=2-2]
			\arrow["a", from=1-1, to=1-2]
			\arrow["b"', from=1-1, to=2-1]
		\end{tikzcd}\]

		\item \emph{Frobenius reciprocity condition}: for any $a:E\to E'$, $l\in L(E)$ and $l'\in L(E')$,
		\[
		\exists_a(L(a)(l') \wedge l)= \exists_a(l) \wedge l'. 
		\] 
	\end{enumerate}	  
\end{prop}	

\begin{proof}
	Let us first show that the given conditions are necessary. By the equivalence of Proposition \ref{propequivlocalicmorphisms}, internal locales in $\Sh({\cal C}, J)$ correspond to localic morphisms $f:{\cal F}\to \Sh({\cal C, J})={\cal E}$, and for any such morphism, the corresponding internal locale $L_{f}$ is given by the composite of the presheaf ${\cal E}^{\textup{op}} \to \Set$ in $\cal E$ (which sends $E$ to $\Sub_{\cal F}(f^{\ast}(E))$), which is a sheaf on $\cal E$ with respect to the canonical topology, with the canonical functor ${\cal C}\to {\cal E}$. The fact that $L_{f}$ is existential, when considered as a fibred site, was observed in the proof of Proposition \ref{propequivlocalicmorphisms}.
	
	Conversely, let us suppose that $L:{\cal C}^{\textup{op}}\to \CAT$ is a fibred site satisfying the conditions of the proposition. The fact that $L$ is $J$-separated implies, by Propositions \ref{propcontainmentGiraudtopology} and \ref{propreflectingfibredsite}, that the existential topology $J^{\textup{ext}}_{L}$ contains the Giraud topology on ${\cal G}(L)$, i.e., that $L$ is an existential fibred site over $({\cal C}, J)$. Thus Proposition \ref{propfibers} ensures that, for any $c\in {\cal C}$, the canonical geometric morphism 
	\[
	\Sh({\cal G}(L), J^{\textup{ext}}_{L})\slash (C_{\pi_{L}})^{\ast}(l(c)) \simeq \Sh({\cal G}^{\textup{ext}}_{c}(L), \tilde{J}_{c}) \to \Sh(L(c), J^{\textup{can}}_{L(c)})
	\]
	is hyperconnected. On the other hand, the topos $\Sh(L(c), J^{\textup{can}}_{L(c)})$ is localic (over $\Set$), whence it can be identified with the localic reflection of the topos $\Sh({\cal G}(L), J^{\textup{ext}}_{L})\slash (C_{\pi_{L}})^{\ast}(l(c))$, in other words, the locale $L(c)$ can be identified with the locale of subterminals of  $\Sh({\cal G}(L), J^{\textup{ext}}_{L})\slash (C_{\pi_{L}})^{\ast}(l(c))$, that is, with the locale of subobjects of $C_{\pi_{L}}^{\ast}(l(c)$ in the existential topos $\Sh({\cal G}(L), J^{\textup{ext}}_{L})$  (for any $c$).
	
	It thus follows that $L$ can be identified with the internal locale $L_{C_{\pi_{L}}}$; in particular, it is an internal locale. 
\end{proof}

\begin{remarks}
	
	\begin{enumerate}[(a)]
		\item If $({\cal C}, J)$ is the canonical site $({\cal E}, J^{\textup{can}}_{\cal E})$ of a Grothendieck topos $\cal E$, $L$ satisfies the conditions of the Proposition if and only if it satisfies the conditions of Definition \ref{definternallocale} (cf. Corollary \ref{thm:existentialcartesian}).
		
		\item If $L:{\cal E}^{\textup{op}}\to \CAT$ is an existential preorder site taking values in the category of frames and frame homomorphisms between them, if $L$ is $J^{\textup{can}}_{\cal E}$-separated then $L$ is a $J^{\textup{can}}_{\cal E}$-sheaf. This formally follows as a consequence of the proof of Proposition \ref{propsitecharinternallocales} (where only the separatedness assumption was used), but can also be easily seen directly. Let us, for instance, show how the sheaf property for a single epimorphism follows from the separatedness property by using the conditions of existential fibred site.
		
		Let us first observe that, for any epimorphism $e:E'\to E$ and any $x\in L(E)$, $\exists_e(L(e)(x))=x$. This amounts to show that the counit $\exists_e \circ L(e) \to 1_{L(E)}$ of the adjunction $(\exists_e \dashv L(e))$ is an isomorphism, equivalently, that $L(e)$ is full and faithful. Since $L(e):L(E)\to L(E')$ is obviously faithful (as the source category is a preorder), and a meet-semilattice homomorphism, it is equivalent to show that $L(e)$ is injective. But this follows immediately from the fact that $L$ is separated with with respect to the canonical topology on $\cal E$.
		
		Now, suppose that $x'\in L(E')$ is an element such that, whenever $h, k$ are such that $e\circ h=e\circ k$, $L(h)(x')=L(k)(x')$; let us consider the element $x:=\exists_{e}(x')$. We want to show that $x'=L(e)(\exists_{e}(x'))$ (in fact, if there is an amalgamation to this matching family, it is necessarily equal to $\exists_e(x')$, since for any $x\in L(E)$, $x=\exists_e(L(e)(x))$). For this, consider the pullback square  	
		\[\begin{tikzcd}
			P & {E'} \\
			{E'} & E
			\arrow["e", two heads, from=2-1, to=2-2]
			\arrow["e", two heads, from=1-2, to=2-2]
			\arrow["{p_1}", two heads, from=1-1, to=1-2]
			\arrow["{p_2}"', two heads, from=1-1, to=2-1]
			\arrow["\lrcorner"{anchor=center, pos=0.125}, draw=none, from=1-1, to=2-2]
		\end{tikzcd}\]
		
		The arrows $p_1$ and $p_2$ are epimorphisms as they are pullbacks of an epimorphism. Moreover, since $e\circ p_1=e\circ p_2$, we have by our hypothesis that $L(p_1)(x')=L(p_2)(x')$. Now, by the Beck-Chevalley condition, $L(e)(\exists_{e}(x'))=\exists_{p_1}(L(p_2)(x'))$. But $L(p_2)(x')=L(p_1)(x')$, whence $$L(e)(\exists_{e}(x'))= \exists_{p_1}(L(p_1)(x'))=x',$$ where the last equality follows from the fact that $p_1$ is an epimorphism (by what we proved above).
		
		\item From the proof of the Propositon, it is clear that one can equivalently characterize the internal locales over $({\cal C}, J)$ as the existential fibred preorder sites over $({\cal C}, J)$ such that, for every $c\in {\cal C}$, $L(c)$ is a frame, $J^{L}_{c}$ is the canonical topology on it, and the canonical morphism $L\to L_{C_{\pi_{L}}}$ is an isomorphism.  
	\end{enumerate}	
\end{remarks}

\subsection{Open existential fibred sites}

Recall that a geometric morphism $f:{\cal F}\to {\cal E}$ is said to be \emph{open} if its inverse image $f^{\ast}$ is a Heyting functor (equivalently, if it satisfies the equivalent conditions of Theorem C3.1.7 \cite{elephant}).

The following result is a generalization of Proposition 1 from section 8 of Chapter 9 of \cite{maclanemoerdijk}; it provides a sufficient condition for a comorphism of sites to induce an open morphism.

\begin{prop}\label{propopen}
	Let $\pi:({\cal D}, K)\to ({\cal C}, J)$ be a comorphism of small-generated sites. If $\pi$ is cover-preserving and moreover satisfies the property that for any object $d\in {\cal D}$ and arrow $\alpha:c\to \pi(d)$ there are a $J$-covering family $\{f_i:c_i \to c\}$ and for each $i\in I$ an arrow $g_i:d_i\to d$ such that $\alpha \circ f_i=\pi(g_i)$. Then the geometric morphism $C_{\pi}:\Sh({\cal D}, K)\to \Sh({\cal C}, J)$ induced by $F$ is open. 
\end{prop}

\begin{proof}
	We can obtain a proof by modifying an argument used in the proof of Proposition 1 from section 8 of Chapter 9 of \cite{maclanemoerdijk}. In fact, the only point in the proof where the hypothesis that the functor $F/d$ be surjective on objects is used is in the proof that the sieve $T_{eu}$ on $c'$ is $J$-covering, at page 509. We shall use a more refined argument allowing us to replace their hypothesis with ours, which is weaker. For any $k\in u^{\ast}(T_{e})$, by definition of $T_{e}$ there are an arrow $v_{k}:c_k \to \pi(d_k)$ and an arrow $h_{k}:\pi_{d_{k}}\to c$ such that $E(h_{k})(e)\in Q(d_k)$ and the following diagram commutes:
	\[\begin{tikzcd}
		{c_k} & {\pi(d_k)} \\
		{c'} & c
		\arrow["{h_k}", from=1-2, to=2-2]
		\arrow["u"', from=2-1, to=2-2]
		\arrow["k"', from=1-1, to=2-1]
		\arrow["{v_k}", from=1-1, to=1-2]
	\end{tikzcd}\]
	
	Now, by our hypothesis, applied to the arrow $v_{k}$, there are a $J$-covering family $Z_{k}$ and, for each $z\in Z_{k}$, an arrow $g_{z}$ such that $v_{k}\circ z=\pi(g_z)$. We shall prove that $T_{E(u)(e)}$ is $J$-covering by showing that it contains the multicomposition of the sieve $u^{\ast}(T_e)$ with the sieves $\{Z_{k} \mid k\in u^{\ast}(T_e)\}$, that is, by verifying that for any arrow $\xi$ such that $\xi=k\circ z$ where $k\in u^{\ast}(T_e)$ and $z\in Z_{k}$, $\xi\in T_{E(u)(e)}$. In fact, this is true since $E(\xi)(E(u)(e))=E(z)(E(k)(E(u)(u)))=E(\pi_{g_{k}})(E(h_{k}(e)))$, which lies in $Q$ by functoriality of $Q$ since $E(h_{k}(e))$ does.  	    
\end{proof}

\begin{prop}
	Let $L$ be an open existential fibred site. Then, for any arrow $f:c\to c'$ in $\cal C$, the geometric morphism
	\[
	\Sh(L(f))\cong C_{\exists_f}:\Sh(L(c), J^{L}_{c}) \to \Sh(L(c'), J^{L}_{c'})  
	\]
	is open. Moreover, for any $c\in {\cal C}$, the geometric morphism
		\[
	\Sh(i_{c})\cong C_{\textup{ext}_{c}}:\Sh({\cal G}^{\textup{ext}}_{c}(L), \tilde{J}_{c}) \to \Sh(L(c), J^{L}_{c})
	\]
	of Proposition \ref{propfibers} is open. 
\end{prop}

\begin{proof}
	This follows by applying Proposition \ref{propopen} to the comorphism of sites $\exists_{f}:(L(c), J^{L}_{c}) \to (L(c'), J^{L}_{c'})$. The hypotheses of the Proposition are satisfied since $\exists_{f}$ is cover-preserving ($L$ being open) and the property in the statement of the Proposition holds by the Relative Frobenius condition. 
\end{proof}

\subsection{Continuous existential fibred sites}

It is useful, in connection with the problem of intrinsically characterizing the existential fibred sites, to introduce the following strenghtening of the notion of open existential fibred site:

\begin{defn}\label{defcontinuousfibredsite}
An existential fibred site over a small-generated site $({\cal C}, J)$ is said to be \emph{continuous} if the following two conditions are satisfied:
\begin{enumerate}[(a)]
	\item for any arrow $a:E'\to E$ in $\cal C$, the functor $\exists_{a}:(L(E'), J^{L}_{E'})\to (L(E), J^{L}_{E})$ is continuous;
	
	\item for any object $(E, x)$ of ${\cal G}(L)$, the functor
	\[
	\xi_{(E, x)}:({\cal G}(L)\slash (E, x), (J^{\textup{ext}}_{L})_{(E, x)}) \to L(E)\slash x
	\]
	sending an object $[(e, \alpha):(E', x')\to (E, x)]$ of ${\cal G}(L)\slash (E, x)$ to the object $[\overline{\alpha}:\exists_{e}(x') \to x]$ of $L(E)\slash x$ (and acting on arrows in the obvious way) is continuous.
\end{enumerate}
\end{defn} 

\begin{ex}
	For any geometric morphism $f:{\cal F}\to {\cal E}$, the existential fibred site $E_{f}$ over $({\cal E}, J^{\textup{can}}_{\cal E})$ is continuous. Indeed, condition (a) is satisfied since for any arrow $a:E' \to E$ in $\cal E$, the functor $\exists_{a}=\Sigma_{f^{\ast}(a)}:({\cal F}\slash f^{\ast}(E'), J^{\textup{can}}_{{\cal F}\slash f^{\ast}(E')} \to ({\cal F}\slash f^{\ast}(E), J^{\textup{can}}_{{\cal F}\slash f^{\ast}(E)})$ given by composition with $f^{\ast}(a)$ is continuous by Corollary 4.47 \cite{denseness}, as it is a morphism of fibrations over $\cal F$, each endowed with the Giraud topology induced by the canonical one on the base topos $\cal F$.   
	
	As to condition (b), it is satisfied again by Corollary 4.47 \cite{denseness} since, for any object $(E, [x:F\to f^{\ast}(E)])$ of ${\cal G}(E_{f})$, the functor $\xi_{(E, x)}: {\cal G}(E_{f})\slash (E, x) \to  ({\cal F}\slash f^{\ast}(E))\slash [x:F\to f^{\ast}(E)] \simeq {\cal F}\slash F$ is clearly a morphism of fibrations over $\cal F$, each endowed with the Giraud topology induced by $J^{\textup{can}}_{\cal F}$.   
\end{ex}

\begin{prop}
	Let $L$ be a continuous existential fibred site over a small-generated site $({\cal C}, J)$. Then, for any object $c$ of $\cal C$, the comorphism of sites
	\[
	\textup{ext}_{c}:({\cal G}^{\textup{ext}}_{c}(L), \tilde{J}_{c})\to (L(c), J^{L}_{c})
	\]
	is continuous.  
\end{prop}

\begin{proof}
Let $c$ be an object of $\cal C$. Then $\textup{ext}_{c}:({\cal G}^{\textup{ext}}_{c}(L), \tilde{J}_{c}) (L(c), J^{L}_{c})$ is continuous if and only if for any object $(f, y)$ of ${\cal G}^{\textup{ext}}_{c}(L)$, the functor 
\[
\textup{ext}_{c}\slash (f, y):({\cal G}^{\textup{ext}}_{c}(L)\slash (f, y), (\tilde{J}_{c})_{(f, y)}) \to (L(c)\slash (\exists_{f}(y)), (J^{L}_{c})_{\exists_{f}(y)})
\]	
is continuous.

But this functor is the composite of the canonical functor
\[
{\cal G}^{\textup{ext}}_{c}(L)\slash (f, y) \to {\cal G}(L)\slash (\dom(f), y)
\]
which is $(\tilde{J}_{c}, J^{\textup{ext}_{L}}_{(\dom(f), y)})$-continuous by Corollary 4.47 \cite{denseness} as it is a morphism of fibrations over ${\cal G}(L)$, each endowed with the Giraud topology induced by the existential topology on ${\cal G}(L)$, followed by the functor
\[
\xi_{(E, x)}:({\cal G}(L)\slash (\dom(f), y), (J^{\textup{ext}}_{L})_{(\dom(f), y)}) \to (L(\dom(f))\slash y, (J^{L}_{\dom(f)})_{y})
\]
of Definition \ref{defcontinuousfibredsite}(b), which is continuous (with respect to the relevant topologies) by our hypotheses, and the functor
\[
\exists_{f}\slash y:L(\dom(f))\slash y \to L(c)\slash \exists_{f}(y) 
\] 
which is also continuous by our hypotheses (as a relativisation of a continuous functor). Therefore, it is continuous, as required.
\end{proof}

We can apply the criterion for a continuous comorphism of sites to induce an equivalence of toposes provided by Proposition 7.17 \cite{denseness} to obtain necessary and sufficient conditions for the canonical morphism of Proposition \ref{propfibers} between the fiber of the associated existential topos and the topos of sheaves on the fiber with the respective topology to be an equivalence. This leads to the following result:

\begin{prop}\label{propcharcontinuousinducingequivalence}
	Let $L$ be a continuous existential fibred site over a small-generated site $({\cal C}, J)$. Then, for any object $c$ of $\cal C$, the following conditions are equivalent:
	\begin{enumerate}[(i)]
		\item the geometric morphism 		
		\[
		\Sh(i_{c})\cong C_{\textup{ext}_{c}}:\Sh({\cal G}^{\textup{ext}}_{c}(L), \tilde{J}_{c}) \to \Sh(L(c), J^{L}_{c})
		\]
		of Proposition \ref{propfibers} is an equivalence;
		
		\item the comorphism of sites  
		\[
		\textup{ext}_{c}:({\cal G}^{\textup{ext}}_{c}(L), \tilde{J}_{c}) \to (L(c), J^{L}_{c})
		\]
		is $\tilde{J}_{c}$-full, $\tilde{J}_{c}$-faithful and $J^{L}_{c}$-dense.
	\end{enumerate}
\end{prop}\qed

\subsection{Intrinsic characterizations}

The following theorem provides intrinsic characterizations of the property of a fibred site to be of the form $E_{f}$ for a geometric morphism $f:{\cal F}\to {\cal E}$:

\begin{thm}
	Let $({\cal C}, J)$ be a small-generated site and $L$ a fibred site over $({\cal C}, J)$. Then the following conditions are equivalent:
	
	\begin{enumerate}[(i)]
		\item $L$ is of the form $E_{f}$ for a geometric morphism $f:{\cal F}\to \Sh({\cal C}, J)$;
		
		\item $L$ is existential, $L(c)$ is a topos, $J(c)$ is the canonical topology on it and the morphism of sites
		\[
		i_{c}:(L(c), J^{L}_{c}) \to ({\cal G}^{\textup{ext}}_{c}(L), \tilde{J}_{c})
		\] 
		is weakly dense (in the sense of Definition 5.4 of \cite{denseness});
		
		\item $L$ is a continuous and existential, $L(c)$ is a topos, $J(c)$ is the canonical topology on it and, for any $c\in {\cal C}$, the comorphism of sites  
		\[
		\textup{ext}_{c}:({\cal G}^{\textup{ext}}_{c}(L), \tilde{J}_{c}) (L(c), J^{L}_{c})
		\]
		is $\tilde{J}_{c}$-full, $\tilde{J}_{c}$-faithful and $J^{L}_{c}$-dense. 
	\end{enumerate}
\end{thm}

\begin{proof}
The equivalence between (i) and (ii) follows from Theorem 5.7 \cite{denseness}, in light of Proposition \ref{propfibers}, while the equivalence between (i) and (iii) follows from Proposition \ref{propcharcontinuousinducingequivalence}.	
\end{proof}

\subsection{Co-Giraud topologies}

It is natural to consider, in the context of the relative sites associated with a localic morphism $f:{\cal F}\to {\cal E}$, the functor
\[
s_{f}:{\cal E}\to \CAT
\]
sending an object $E$ of $\cal E$ to the poset $\Sub_{\cal F}(f^{\ast}(E))$ and an arrow $e:E\to E'$ to the map $\exists_{f^{\ast}(e)}:\Sub_{\cal F}(f^{\ast}(E)) \to \Sub_{\cal F}(f^{\ast}(E'))$. In fact, ${\cal G}(E_{f})$ coincides with the covariant Grothendieck construction of $s_{f}$ (since, for each arrow $g:F\to F'$ in $\cal F$, the pullback functor $g^{\ast}:\Sub_{\cal F}(F')\to \Sub_{\cal F}(F)$ is right adjoint to the functor $\exists_{g}:\Sub_{\cal F}(F)\to \Sub_{\cal F}(F')$), whence the canonical projection functor
\[
p_{E_{f}}:{\cal G}(E_{f}) \to {\cal E}
\]
is both a fibration and a opfibration, and all the morphisms of the form 
\[
e:(E, l) \to (E', \exists_{f^{\ast}(E)}(l))
\] 
are co-cartesian with respect to the opfibration $s_{f}$.  

We shall say that \emph{co-relative site} over a base site $({\cal C}, J)$ is an opfibration $q:{\cal D}\to {\cal C}$ with a Grothendieck topology $H$ on $\cal D$ which contains the \emph{co-Giraud topology} induced by $J$ on $\cal D$ (that is, the Grothendieck topology on $\cal D$ whose covering families are those which contain families of co-cartesian arrows obtained by lifting $J$-covering families along $p$). 

We shall say that a co-relative topology is \emph{co-orthogonally generated} if it is generated by a collection of families which are either vertical or entirely consisting of co-cartesian arrows.

\begin{prop}
	Let $\cal E$ be a Grothendieck topos and $L$ an internal locale in $\cal E$. Then the canonical relative topology of $L$ (over the canonical site of $\cal E$) is co-orthogonally generated as a co-relative site over ${\cal E}$ (but not necessarily orthogonally generated over $\cal E$). 
\end{prop} 

\begin{proof}
	The thesis follows at once by noticing that, given a $J_{L}^{\textup{can}}$-covering family	
	\[
	\{e_i:(E_i, l_i)\to (E, l) \mid i\in I\}
	\] 
	we have, for each $i\in I$, the following factorization of $e_{i}$ in ${\cal G}(L)$ as a co-cartesian arrow followed by a vertical arrow:
	\[\begin{tikzcd}
		{(E_i, l_i)} & {(E, \exists_{e_i}(l_i))} \\
		& {(E, l)}
		\arrow["{1_E}", from=1-2, to=2-2]
		\arrow["{e_i}", from=1-1, to=1-2]
		\arrow["{e_i}"', from=1-1, to=2-2]
	\end{tikzcd}\]
	
	So the family $\{e_i:(E_i, l_i)\to (E, l) \mid i\in I\}$ can be written as the multicomposite of the covering family $\{ (E, \exists_{e_i}(l_i)) \to (E, l)  \mid i\in I \}$ of vertical arrows with the single co-cartesian covering arrows $e_i:(E_i, l_i) \to (E, \exists_{e_i}(l_i))$ (for $i\in I$). Thus the topology $J^{\textup{ext}}_{L}$ is co-orthogonally generated, as desired.
\end{proof}

\begin{remark}
	In a completely analogous way, one can show that, for any geometric morphism $f:{\cal F}\to {\cal E}$, the existential fibred site $({\cal G}(E_{f}), J^{\textup{ext}}_{E_{f}})$ associated with $f$ is co-orthogonally generated over $({\cal E}, J^{\textup{can}}_{\cal E})$.
\end{remark}

The representation of relative toposes as existential toposes associated with existential fibred sites generalizes the construction of toposes of sheaves on a locale. In other words, this formalism allows us to regard toposes \textit{as if they were} toposes of sheaves on a locale. This can be a powerful point of view, both geometrically and logically, by the relative Diaconescu's theorem (yielding the fullness and faithfulness of the embedding, analogous to the full and faithful embedding of the category of locales into that of Grothendieck toposes). The notion of existential topology is fundamental as it allows one to reconstruct a relative topos directly from the associated fibred site. 

The logical counterpart of this point of view on relative toposes is related to the doctrinal approach to categorical logic (cf. \cite{hyperdoctrines}) and would deserve to be thoroughly investigated in connection with the development of relative geometric logic.

\section{The completion of a fibred preorder site}\label{sec:completionfibredpreordersite}

\begin{defn}
	\begin{enumerate}[(a)]
		\item 	A \emph{fibered preorder} over a category $\cal C$ is a pseudofunctor ${\cal C}^{\textup{op}}\to \CAT$ which takes values in the full subcategory of $\CAT$ on preorder categories. 
		
		\item A \emph{fibered preordered site} over a small-generated site $({\cal C}, J)$ is a fibered preorder $\mathbb D$ over $\cal C$ together with a Grothendieck topology on ${\cal G}({\mathbb D})$ which contains the Giraud topology $M_{J}^{p_{\mathbb D}}$.  
	\end{enumerate}	
\end{defn}

Let ${\mathbb P}:{\cal C}^{\textup{op}} \to \CAT$ be a fibered preorder, and $J$ a Grothendieck topology on $\cal C$. Then for any relative site $p:({\cal G}({\mathbb P}), K)\to ({\cal C}, J)$ over $({\cal C}, J)$, the geometric morphism
\[
C_p:\Sh({\cal G}({\mathbb D}), K)\to \Sh({\cal C}, J)
\]
is localic (by Proposition 7.11 \cite{denseness}). In this section we shall compute the internal locale $L_{C_p}$ corresponding to it. We shall see that it represents a sort of \ac completion' of ${\mathbb D}$ with respect to the Grothendieck topology $K$. 

Recall from \cite{CaramelloStone} that, given a preorder site $({\cal C}, J)$, the frame $\textup{Id}_{J}({\cal C})$ of $J$-ideals on $\cal C$ is a frame which, together with the canonical map ${\cal C}\to \textup{Id}_{J}({\cal C})$ (sending any element of $\cal C$ to the $J$-closure of the principal ideal generated by $c$), satisfies the following universal property: every order-preserving map ${\cal C} \to F$ to a frame $F$ which is filtering and sends $J$-covering families to joins in $F$ can be extended uniquely along the canonical map ${\cal C}\to \textup{Id}_{J}({\cal C})$ to a frame homomorphism $\textup{Id}_{J}({\cal C}) \to F$. The particular case of this result for meet-semilattices was established in \cite{johnstone.stonespaces} and played a key role in the development of formal topology. 

In this section, we are going to establish the fibered generalization of this theorem: preorder sites are replaced by fibred preorder sites, frames by internal frames, morphisms of preorder sites by morphisms of fibred preorder sites and morphisms of frames by morphisms of internal frames.

\newcommand{\vin}{\rotatebox[origin=c]{-90}{$\mathlarger{\in}$}}


$$
\resizebox{\textwidth}{!}{	
\begin{tikzcd}[row sep=small, column sep=0.02pt, ampersand replacement=\&]
	\&\&\& {\textit{\footnotesize Localic morphism}} \\
	\&\& {\textbf{Sh}({\cal G}({\mathbb P}), K)} \&\& {\textbf{Sh}_{\textbf{Sh}({\cal C}, J)}(L^{K}_{\mathbb P})} \\
	\&\&\& {\textbf{Sh}({\cal C}, J)} \\
	{({\cal G}({\mathbb P}), K)} \&\&\&\&\&\& {L^{K}_{\mathbb P}} \\
	{({\cal C}, J)} \&\&\&\&\&\& {\textbf{Sh}({\cal C},J)} \\
	{\textit{\footnotesize Relative preorder site}} \&\&\&\&\&\& {\textit{\footnotesize Internal locale}}
	\arrow["\vin", shift right=3, draw=none, from=4-7, to=5-7]
	\arrow["{p_{\mathbb P}}", from=4-1, to=5-1]
	\arrow["{C_{{\cal G}(\eta_{\mathbb P})}}", shift left=2, from=2-3, to=2-5]
	\arrow["{\textbf{Sh}({\cal G}(\eta_{\mathbb P}))}", shift left=2, from=2-5, to=2-3]
	\arrow["{C_{p_{\mathbb P}}}"', shift right=2, from=2-3, to=3-4]
	\arrow[from=2-5, to=3-4]
	\arrow["\sim", shift right=2, draw=none, from=2-3, to=2-5]
	\arrow[shift right=3,  bend left=18, shorten <=8pt, dashed, from=4-1, to=2-3]
	\arrow[bend left=18, dashed, no head, from=2-5, to=4-7]
\end{tikzcd}}
$$

The following \ac bridge' shows that the construction of the internal locale $L^{K}_{\mathbb P}$ of a relative preorder site $p_{\mathbb P}:({\cal G}({\mathbb P}), K) \to ({\cal C}, J)$ actually yields a reflection of the category of relative preorder sites over $({\cal C}, J)$ into the category of internal frames to the topos $\Sh({\cal C}, J)$:

$$
\resizebox{\textwidth}{!}{	
\begin{tikzcd}[row sep=tiny, column sep=tiny, ampersand replacement=\&]
	\&\& {\textit{\footnotesize Morphism over }} \Sh({\cal C}, J) \\[-10pt]
	\&\& \textit{\footnotesize from }\textbf{Sh}_{\textbf{Sh}({\cal C}, J)}(L')\,\textit{\footnotesize to}\vspace{0.3cm}  \\[7pt]
	\& {\textbf{Sh}({\cal G}({\mathbb P}), K)} \&\& {\textbf{Sh}_{\textbf{Sh}({\cal C}, J)}(L^{K}_{\mathbb P})} \\[5pt]
	\&\& {\textbf{Sh}({\cal C}, J)} \\
	{({\cal G}({\mathbb P}), K)} \&\&\&\& {L^{K}_{\mathbb P}} \\[5pt]
	{({\cal C}, J)} \&\&\&\& {\textbf{Sh}({\cal C},J)} \\
	{\textit{\footnotesize Morphism of relative sites} } \&\&\&\& {\textit{\footnotesize Locale morphism } } \\[-10pt]
	{p_{\mathbb P} \to p_{{\cal G}(L')}} \&\&\&\& {L'\to L^{K}_{\mathbb P}}
	\arrow["\vin", shift right=3, draw=none, from=5-5, to=6-5]
	\arrow["{p_{\mathbb P}}", from=5-1, to=6-1]
	\arrow["{C_{{\cal G}(\eta_{\mathbb P})}}", shift left=2, from=3-2, to=3-4]
	\arrow["{\textbf{Sh}({\cal G}(\eta_{\mathbb P}))}", shift left=2, from=3-4, to=3-2]
	\arrow["{C_{p_{\mathbb P}}}"', shift right=2, from=3-2, to=4-3]
	\arrow[from=3-4, to=4-3]
	\arrow["\sim", shift right=2, draw=none, from=3-2, to=3-4]
	\arrow[shift right=3, bend left=18, shorten <=6pt, dashed, no head, from=5-1, to=3-2]
	\arrow[bend left=18, dashed, no head, from=3-4, to=5-5]
\end{tikzcd}}
$$

Indeed, by Theorem \ref{thm:RelativeDiaconescuCartesian} (cf. Remark \ref{remrelDiaconescu}), the category of morphisms over $\Sh({\cal C}, J)$ from $\Sh_{\Sh({\cal C}, J)}(L')$ to $(\Sh({\cal G}({\mathbb P}), K))\to \Sh({\cal C}, J)$ is equivalent to the category $$\textup{\bf MorRelSites}_{({\cal C}, J)}((p_{\mathbb P}:({\cal G}({\mathbb P}), K)) \to ({\cal C}, J), (p_{L}({\cal G}(L), J^{\textup{ext}}_{L}) \to ({\cal C}, J)))$$ of morphisms of relative sites over $({\cal C}, J)$ from $p_{\mathbb P}$ to $p_{{\cal G}(L)}$ and natural transformations between them. 

On the other hand, by Theorem \ref{cor:fullandfaithfulnessinternallocales}, the category of morphisms over $\Sh({\cal C}, J)$ from $\Sh_{\Sh({\cal C}, J)}(L')$ to $\Sh_{\Sh({\cal C}, J)}(L^{K}_{\mathbb P})$ is equivalent to the opposite of the category $\textup{\bf Frm}_{\Sh({\cal C}, J)}$ of internal frames to the topos $\Sh({\cal C}, J)$ and internal frame homomorphisms between them.

Summarizing, we have the following result:

\begin{thm}
	Let $({\cal C}, J)$ be a small-generated site, $\textup{\bf MorRelSites}_{({\cal C}, J)}$ the $2$-category of relative preorder sites over $({\cal C}, J)$ and morphisms of relative sites between them and $\textup{\bf Frm}_{\Sh({\cal C}, J)}$ the $2$-category of internal frames to the topos $\Sh({\cal C}, J)$ and internal frame homomorphisms between them. Then the functors 
	\[
	(p_{\mathbb P}:({\cal G}({\mathbb P}), K) \to ({\cal C}, J)) \mapsto L^{K}_{\mathbb P}
	\]
	and 
	\[
	L \mapsto ({\cal G}(L), J^{\textup{ext}}_{L})
	\]
	are $2$-adjoint to each other (the former on the left and the latter on the right) and define a reflection of (since the former functor identifies $\textup{\bf MorRelSites}_{({\cal C}, J)}$ with a full subcategory of $\textup{\bf Frm}_{\Sh({\cal C}, J)}$). 
	
	The unit of this reflection is the dense bimorphism of sites $\eta_{{\cal G}({\mathbb P})}$ of Proposition \ref{prop:unitfibredpreorder}, which induces the equivalences
	\[\begin{tikzcd}
		{\Sh({\cal G}({\mathbb P}), K)} && {\Sh({\cal G}(L^{K}_{\mathbb P}), J^{\textup{ext}}_{L^{K}_{\mathbb P}})} \\
		& {\Sh({\cal C}, J)} \\
		\\
		&&& {}
		\arrow["{C_{{\cal G}(\eta_{\mathbb P})}}", shift left=2, from=1-1, to=1-3]
		\arrow["{\textup{\bf Sh}({\cal G}(\eta_{\mathbb P}))}", shift left=2, from=1-3, to=1-1]
		\arrow["{C_{p_{\mathbb P}}}"', shift right=2, from=1-1, to=2-2]
		\arrow[from=1-3, to=2-2]
		\arrow["\sim", shift right=3, draw=none, from=1-1, to=1-3]
	\end{tikzcd}\]	
\end{thm}\qed

\begin{defn}
	We shall call the internal locale $L^{K}_{\mathbb P}$ the \emph{fibred ideal completion} of the relative preorder site $p_{\mathbb P}:({\cal G}({\mathbb P}), K) \to ({\cal C}, J)$.
\end{defn}

\begin{remark}
	In the case $\Sh({\cal C}, J)$ being the topos $\Set$ of sets and ${\cal C}$ being the one-object and one-arrow category with the trivial topology on it, the locale $L^{K}_{\mathbb P}$ coincides with the locale $\textup{Id}_{K}({\mathbb P})$ of $K$-ideals on the preorder ${\mathbb P}$.
	
	The equivalence in the theorem is the fibred generalization of the equivalence 
	\[
	\Sh({\mathbb P}, K)\simeq \Sh(\textup{Id}_{K}({\mathbb P}), J^{\textup{can}}_{\textup{Id}_{K}({\mathbb P})}).
	\]
\end{remark}

The following proposition gives an explicit description of the canonical functor of Proposition \ref{propbidense} in the case of a fibred preordered site.	

\begin{prop}\label{prop:unitfibredpreorder}
	Let $({\mathbb P}, K)$ be a fibered preordered site over a small-generated site $({\cal C}, J)$. Then the canonical functor $$\eta_{p_{\mathbb P}}:{\cal G}({\mathbb P}) \to (1_{\Sh({\cal G}({\mathbb P}), K)} \downarrow^{\textup{Sub}} C_{p_{\mathbb P}}^{\ast})={\cal G}(L_{C_{p_{\mathbb P}}})$$ of Proposition \ref{propbidense} is the result of applying the Grothendieck's construction to the indexed functor
	\[
	\eta_{\mathbb P}:{\mathbb P} \to L_{C_{p_{\mathbb P}}}
	\] 
	described as follows:
	
	\begin{itemize}
		\item For any $c\in {\cal C}$, $L_{C_{p_{\mathbb P}}}(c)$ identifies with the frame $$\textup{ClSub}^{K}_{[{\cal G}({\mathbb P})^{\textup{op}}, \Set]}(\Hom_{\cal C}(p_{\mathbb P}(-), c))$$ of $K$-closed subobjects in $[{\cal G}({\mathbb P})^{\textup{op}}, \Set]$ of the presheaf $\Hom_{\cal C}(p_{\mathbb P}(-), c)$.
		
		\item The indexed functor $\eta_{\mathbb P}$ acts at an object $c\in {\cal C}$ as the functor 
		\[
		\eta_{\mathbb P}(c):{\mathbb P}(c) \to L_{C_{p_{\mathbb P}}}(c)=\textup{ClSub}^{K}_{[{\cal G}({\mathbb P})^{\textup{op}}, \Set]}(\Hom_{\cal C}(p_{\mathbb P}(-), c))
		\]
		sending any element $x\in {\mathbb P}(c)$ to the $K$-closure of the subfunctor of $\Hom_{\cal C}(p_{\mathbb P}(-), c)$ sending any object $(c', x')$ of ${\cal G}({\mathbb P})$ to the subset $S_{(c', x')}\subseteq \Hom_{\cal C}(p_{\mathbb P}((c', x')), c)=\Hom_{\cal C}(c', c)$ consisting of the arrows $g:c'\to c$ such that $x'\leq L(g)(x)$.   
	\end{itemize} 
\end{prop}	

\begin{proof}
	By definition, $L_{C_{p_{\mathbb P}}}(c)=\Sub_{\Sh({\cal G}({\mathbb P}), K)}(C_{p_{\mathbb P}}^{\ast}(l(c)))$. Now,
	\[
	C_{p_{\mathbb P}}^{\ast}(l(c))=C_{p_{\mathbb P}}^{\ast}(a_{J}(y_{\cal C}(c)))=a_{K}(p_{\mathbb P}^{\ast}(y_{\cal C}(c)))
	\] 
	(since $p_{{\mathbb P}}$ is a comorphism of sites $({\cal G}({\mathbb P}), K)\to ({\cal C}, J)$).
	
	So, by Proposition 2.3 of \cite{denseness},
	\[
	L_{C_{p_{\mathbb P}}}(c)\cong \textup{ClSub}^{K}_{[{\cal G}({\mathbb P})^{\textup{op}}, \Set]}(\Hom_{\cal C}(p_{\mathbb P}(-), c)),
	\]
	as required.
\end{proof}	

\begin{remark}
	Notice the independence of the above constructions from the topology $J$ on the base category $\cal C$; in fact, as one can see in the proof of the proposition, the reason for this is that $p_{{\mathbb P}}$ is a comorphism of sites $({\cal G}({\mathbb P}), K)\to ({\cal C}, J)$ and hence $C_{p_{\mathbb P}}^{\ast}  \circ a_{J}\cong a_{K}\circ p_{\mathbb P}^{\ast}$. 
\end{remark}

\section{The hyperconnected-localic factorization in terms of internal locales}\label{sec:hyperconn}

As shown by the following result, the hyperconnected-localic factorization of a geometric morphism can be naturally described in terms of the internal locale canonically associated with it. 

\begin{thm}\label{hyperconnlocalicinternallocale}
	Let $f:{\cal F}\to {\cal E}$ be a geometric morphism. Then the functor $L_{f}:{\cal E}^{\textup{op}} \to \textup{\bf Cat}$ sending each object $E$ of $\cal E$ to the subobject lattice $\Sub_{\cal F}(f^{\ast}(E))$, and each arrow $e:E' \to E$ in $\cal E$ to the pullback functor $e^{\ast}:\Sub_{\cal F}(f^{\ast}(E)) \to \Sub_{\cal F}(f^{\ast}(E'))$, is an internal locale to $\cal E$. 
	
	Let $i_{f}:{\cal E}\to {\cal G}(L_{f})$ be the functor sending any object $E$ of $\cal E$ to the object $(E, 1_{f^{\ast}(E)})$ of ${\cal G}(L_{f})$. 
	
	Then $p_{L}\dashv i_{f}$, $i_{f}$ is a morphism of sites $({\cal E}, J^{\textup{can}}_{\cal E}) \to ({\cal G}(L_{f}), J^{\textup{ext}}_{L_{f}})$ and $\pi_{L}$ is a comorphism of sites $({\cal G}(L_{f}), J^{\textup{ext}}_{L_{f}}) \to ({\cal E}, J^{\textup{can}}_{\cal E})$ inducing the same geometric morphism $\Sh(i_{f})\cong C_{\pi_{\cal E}}:\Sh_{\cal E}(L_{f}) \to {\cal E}$, which is the localic reflection of $f$ (that is, the localic part of the hyperconnected-localic factorization of $f$).
	
	The canonical projection functor $\pi_{\cal F}:{\cal G}(L_{f})\to {\cal F}$ is a morphism of sites $({\cal G}(L_{f}), J^{\textup{ext}}_{L_{f}}) \to ({\cal F}, J^{\textup{can}}_{\cal F})$, and the geometric morphism $\Sh(\pi_{\cal F}):{\cal F}\to \Sh_{\cal E}(L_{f})$ it induces is the hyperconnected part of the hyperconnected-localic factorization of $f$.
	
	
	\[\begin{tikzcd}
		{\textup{\bf Sh}({\cal F}, J^{\textup{can}}_{\cal F})\simeq {\cal F}} && {{\cal E}\simeq \textup{\bf Sh}({\cal E}, J^{\textup{can}}_{\cal E}) } \\
		\\
		& {\textup{\bf Sh}_{\cal E}(L_{f}):=\textup{\bf Sh}({\cal G}(L_{f}), J^{\textup{ext}}_{L_{f}})}
		\arrow["f", from=1-1, to=1-3]
		\arrow["{\textup{\bf Sh}(\pi_{\cal F})}"', from=1-1, to=3-2]
		\arrow["{C_{p_{L}}\cong\textup{\bf Sh}(i_{f}) }"', from=3-2, to=1-3]
		\arrow["hyperconnected"{sloped}, shift left=1, draw=none, from=1-1, to=3-2]
		\arrow["{localic}"{sloped}, shift left=1, draw=none, from=3-2, to=1-3]
	\end{tikzcd}\]	   
\end{thm}

\begin{proof}
	The geometric morphism $$C_{\pi_{\cal E}}:\Sh_{\cal E}(L_{f})=\Sh({\cal G}(L_{f}), J^{\textup{ext}}_{L_{f}}) \to \Sh({\cal E}, J^{\textup{can}}_{\cal E})\simeq {\cal E}$$ is localic (by Proposition 7.11 \cite{denseness}).
	
	The canonical projection functor $\pi_{\cal F}:{\cal G}(L_{f})\to {\cal F}$ is a morphism of sites $({\cal G}(L_{f}), J^{\textup{ext}}_{L_{f}}) \to ({\cal F}, J^{\textup{can}}_{\cal F})$ since it is equal to the composite of the inclusion functor $({\cal G}(L_{f}), J^{\textup{ext}}_{L_{f}}) \to ((1_{\cal F}\downarrow f^{\ast}), J_{f})$, which was observed in section \ref{sec:localicmorphisms} to be a morphism of sites (cf. Proposition \ref{prop:densenessrelsitelocalicmorphism} therein), and the canonical projection functor $(1_{\cal F}\downarrow f^{\ast}) \to {\cal F}$, which is a morphism of sites $((1_{\cal F}\downarrow f^{\ast}), J_{f}) \to ({\cal F}, J^{\textup{can}}_{\cal F})$.
	
	The morphism $\Sh(\pi_{\cal F}):{\cal F}\to \Sh_{\cal E}(L_{f})$ induced by it is hyperconnected since $\pi_{\cal F}:({\cal G}(L_{f}), J^{\textup{ext}}_{L_{f}}) \to ({\cal F}, J^{\textup{can}}_{\cal F})$ satisfies the conditions of Proposition \ref{prophyperconnectedsites}. Indeed, the fact that it is closed-sieve lifting is obvious, while the fact that $\pi_{\cal F}$ is cover-reflecting follows from the fact that, as it is readily seen, the property of a family of arrows 
	\[
	\{(f_i, e_i):(F_i, E_i, \alpha_i:F_i\mono f^\ast(E_i))\to (F, E, \alpha:F\mono f^{\ast}(E)) \mid i\in I\}
	\]
	in ${\cal G}(L_{f})$ to be sent by $\pi_{\cal F}$ to an epimorphic family (i.e. to be such that the family $\{f_i \mid i\in I\}$ is epimorphic in $\cal F$) can be reformulated as the condition that $\bigcup_{i\in I}\exists_{f^{\ast}(e_i)}(\alpha_{i})=1_{f^{\ast}(E)}$ in $\Sub_{\cal F}(f^{\ast}(E))=L_{f}(E)$ (i.e. as the condition that it should be $J^{\textup{ext}}_{L_{f}}$-covering).	
\end{proof}

The following proposition, which is an easy corollary of Theorem \ref{hyperconnlocalicinternallocale} gives a site-theoretic description of the above construction in the case of geometric morphisms induced by a morphisms of sites.

\begin{prop}
	Let $A:({\cal C}, J) \to ({\cal D}, K)$ be a morphism of sites. Then the functor
	\[
	L_{A}:{\cal C}^{\textup{op}} \to \textup{\bf Cat}
	\] 	
	sending any object $c$ of $\cal C$ to the frame $\textup{ClSv}^{K}(A(c))$ of $K$-closed sieves on $A(c)$ and an arrow $f:c\to c'$ in $\cal C$ to the operation of pulling back ($K$-closed) sieves along $A(f)$ is an internal locale in $\Sh({\cal C}, J)$.
	
	Let $i_{A}:{\cal C}\to {\cal G}(L_{A})$ be the functor sending any object $c$ of $\cal C$ to the object $(c, M_{A(c)})$ of ${\cal G}(L_{A})$, where $M_{A(c)}$ is the maximal sieve on $A(c)$, and any arrow $f:c\to c'$ in $\cal C$ to the arrow $A(f):(c, M_{A(c)}) \to (c', M_{A(c')})$ in ${\cal G}(L_{A})$.
	
	Then $i_{A}$ is a morphism of sites $({\cal C}, J) \to ({\cal G}(L_{A}), J^{\textup{ext}}_{L_{A}})$, right adjoint to the comorphism of sites $p_{L(A)}:({\cal G}(L_{A}), J^{\textup{ext}}_{L_{A}})\to ({\cal C}, J)$ given by the canonical projection functor. 
	
	Let $z_{A}:{\cal G}(L_{A})\to {\cal D}^{s}_{K}$ be the canonical functor towards the category of Corollary 2.16 \cite{denseness}, and $C^{s}_{K}$ the Grothendieck topology on ${\cal D}^{s}_{K}\hookrightarrow \Sh({\cal D}, K)$ induced by the canonical topology on $\Sh({\cal D}, K)$. Then $z_{A}$ yields a morphism of sites $z_{A}:({\cal G}(L_{A}), J^{\textup{ext}}_{L_{A}})\to ({\cal D}^{s}_{K}, C^{s}_{K})$ inducing an hyperconnected morphism $\Sh(z_{A}):\Sh({\cal D}^{s}_{K}, C^{s}_{K}) \to \Sh({\cal G}(L_{A}), J^{\textup{ext}}_{L_{A}})$.
	
	So, the factorization of $A$ as $z_{A}\circ i_{A}$ induces the hyperconnected-localic factorization of the geometric morphism $\Sh(A)$: 

$$
\resizebox{\textwidth}{!}{	
	\begin{tikzcd}[ampersand replacement=\&]
		{\Sh({\cal D}^{s}_{K}, C^{s}_{K})\simeq \Sh({\cal D}, K)} \&\& \textup{\bf Sh}({\cal C}, J) \\
		\\
		\& {\textup{\bf Sh}_{\Sh({\cal C}, J)}(L_{A}):=\textup{\bf Sh}({\cal G}(L_{A}), J^{\textup{ext}}_{L_{A}})}
		\arrow["\Sh(A)", from=1-1, to=1-3]
		\arrow["{\textup{\bf Sh}(z_{A})}"', from=1-1, to=3-2]
		\arrow["{C_{p_{L_{A}}}\cong\textup{\bf Sh}(i_{A}) }"', from=3-2, to=1-3]
		\arrow["hyperconnected"{sloped}, shift left=1, draw=none, from=1-1, to=3-2]
		\arrow["{localic}"{sloped}, shift left=1, draw=none, from=3-2, to=1-3]
	\end{tikzcd}}
$$
\end{prop}\qed	

\begin{remark}
	Let $\cal F$ be a Grothendieck topos, $(\textbf{1}, T)$ the site of definition for the topos $\Set$ constisting of the one-object and one-arrow category $\textbf{1}$ with the trivial topology $T$, and $A$ the morphism of sites $(\textbf{1}, T)\to ({\cal F}, J^{\textup{can}}_{\cal F})$ picking the terminal object $1$ of $\cal F$ (where $\textup{1}$ is the one-object and one-arrow category and $T$ is the trivial topology on it). Then $L_{A}$ identifies with the frame of subterminal objects $\Sub_{\cal F}(1)$ of $1$ and, under this identification, the functor $z_{A}$ corresponds to the forgetful (inclusion) functor $\Sub_{\cal F}(1) \to {\cal F}$. Notice that this functor has a left adjoint, sending an object $F$ of $\cal F$ to the image of the unique morphism $F\to 1$; this left adjoint is therefore is a comorphism of sites $({\cal F}, J^{\textup{can}}_{\cal F}) \to (\Sub_{\cal F}(1), J^{\textup{can}}_{\Sub_{\cal F}(1)})$, which induces the hyperconnected part of the hyperconnected-localic factorization of the unique geometric morphism ${\cal F}\to \Set\simeq \Sh(\textbf{1}, T)$. This comorphism of sites can be profitably used to represent the canonical hyperconnected morphism ${\cal F} \to \Sh((\Sub_{\cal F}(1))$ in situations where one does not dispose of a site of definition $({\cal D}, K)$ for $\cal F$ whose underlying category is closed in $\cal F$ under subobjects and hence does not have a morphism of sites $z_{A}$ taking values in $({\cal D}, K)$; indeed, a comorphism of sites $({\cal D}, K) \to (\Sub_{\cal F}(1), J^{\textup{can}}_{\Sub_{\cal F}(1)})$ inducing that morphism can be obtained by taking the restriction to $\cal D$ of the canonical comorphism to the subterminals.     
\end{remark}

In light of the closure under subobjects of geometric syntactic categories in the respective classifying toposes, it is interesting to apply the proposition in the context of interpretations of geometric theories.

\begin{prop}\label{prop:internalocaleinterpretations}
	Let $I:{\cal C}_{\mathbb T}\to {\cal C}_{\mathbb S}$ be an interpretation of a geometric theory $\mathbb T$ into a geometric theory ${\mathbb S}$. 
	
	Then the functor
	\[
	L_{I}:{{\cal C}_{\mathbb T}}^{\textup{op}} \to \textup{\bf Cat}
	\] 	
	sending any object $\{\vec{x}. \phi\}$ of ${\cal C}_{\mathbb T}$ to the frame $L_{I}(I(\{\vec{x}. \phi\}))=\Sub_{{\cal C}_{\mathbb S}}(I(\{\vec{x}. \phi\}))$ of $\mathbb S$-provable equivalence classes of geometric formulas-in-context which provably imply the formula-in-context $I(\{\vec{x}. \phi\})$ and any arrow $[\theta]:\{\vec{x}. \phi\}\to \{\vec{y}. \psi\}$ in ${\cal C}_{\mathbb T}$ to the operation of pulling back subobjects along $I([\theta])$ is an internal locale in $\Sh({\cal C}_{\mathbb T}, J_{\mathbb T})$.
	
	The functor $i_{I}:{\cal C}_{\mathbb T}\to {\cal G}(L_{I})$ sending each object to the corresponding maximal subobject is a morphism of sites $({\cal C}_{\mathbb T}, J_{\mathbb T}) \to ({\cal G}(L_{I}), J^{\textup{ext}}_{L_{I}})$, right adjoint to the comorphism of sites $p_{L(I)}:({\cal G}(L_{I}), J^{\textup{ext}}_{L_{I}})\to ({\cal C}_{\mathbb T}, J_{\mathbb T})$ given by the canonical projection functor. 
	
	The canonical functor $z_{I}:{\cal G}(L_{I})\to {\cal C}_{\mathbb S}$ taking the domain of subobjects yields a morphism of sites $z_{I}:({\cal G}(L_{I}), J^{\textup{ext}}_{L_{I}})\to ({\cal C}_{\mathbb S}, J_{\mathbb S})$ inducing an hyperconnected morphism $\Sh(z_{I}):\Sh({\cal C}_{\mathbb S}, J_{\mathbb S}) \to \Sh({\cal G}(L_{I}), J^{\textup{ext}}_{L_{I}})$. 
	
	So, the factorization of $I$ as $z_{I}\circ i_{I}$ induces the hyperconnected-localic factorization of the geometric morphism $\Sh(I)$: 
	
	\[\begin{tikzcd}
		{\Sh({\cal C}_{\mathbb S}, J_{\mathbb S})} && \textup{\bf Sh}({\cal C}_{\mathbb T}, J_{\mathbb T}) \\
		\\
		& {\textup{\bf Sh}_{\Sh({\cal C}_{\mathbb T}, J_{\mathbb T})}(L_{I}):=\textup{\bf Sh}({\cal G}(L_{I}), J^{\textup{ext}}_{L_{I}})}
		\arrow["\Sh(I)", from=1-1, to=1-3]
		\arrow["{\textup{\bf Sh}(z_{I})}"', from=1-1, to=3-2]
		\arrow["{C_{p_{L_{I}}}\cong\textup{\bf Sh}(i_{I}) }"', from=3-2, to=1-3]
		\arrow["hyperconnected"{sloped}, shift left=1, draw=none, from=1-1, to=3-2]
		\arrow["{localic}"{sloped}, shift left=1, draw=none, from=3-2, to=1-3]
	\end{tikzcd}\]	   
	
	In particular, if $I$ is the interpretation associated with a localic expansion of theories, or more generally an interpretation inducing a localic geometric morphism, the morphism $\Sh(z_{I})$ is an equivalence  
	\[
	\Sh({\cal C}_{\mathbb S}, J_{\mathbb S})\simeq \textup{\bf Sh}({\cal G}(L_{I}), J^{\textup{ext}}_{L_{I}})=\textup{\bf Sh}_{\Sh({\cal C}_{\mathbb T}, J_{\mathbb T})}(L_{I})
	\]
	and hence $\textup{\bf Sh}({\cal G}(L_{I}), J^{\textup{ext}}_{L_{I}})$ is an alternative site of definition for the classifying topos of $\mathbb S$.
\end{prop}\qed

\begin{remark}
	If we take ${\mathbb T}$ to be the (cartesian) empty theory ${\mathbb O}_{\Sigma_{\mathbb S}}$ over the signature consisting precisely of the sorts of the signature $\Sigma_{\mathbb S}$ of $\mathbb S$ (with no function or relation symbols) then $\mathbb S$ is a localic expansion of ${\mathbb O}_{\Sigma_{\mathbb S}}$ and hence, if we take $I$ to be the associated interpretation ${\cal C}^{\textup{cart}}_{{\mathbb O}_{\Sigma_{\mathbb S}}} \to {\cal C}_{\mathbb S}$ of ${\mathbb O}_{\Sigma_{\mathbb S}}$ into $\mathbb S$, we recover as site $\textup{\bf Sh}({\cal G}(L_{I}), J^{\textup{ext}}_{L_{I}})$ the alternative syntactic site for the classifying topos of $\mathbb S$ introduced by J. Wrigley in \cite{WrigleyToposesOnline} (we refer to his work for a thorough analysis of this construction). 
\end{remark}

\bibliography{biblio}{}
\bibliographystyle{abbrv}

\vspace{1cm}

\textbf{Acknowledgements:} Part of this work has been written during a visit of the author at the Lagrange Mathematics and Computing Research Center in Paris, which we thank for the excellent working conditions. 

\vspace{1cm}

\textsc{Olivia Caramello} 

\vspace{0.2cm}
{\small \textsc{Dipartimento di Scienza e Alta Tecnologia, Universit\`a degli Studi dell'Insubria, via Valleggio 11, 22100 Como, Italy.}\\
	\emph{E-mail address:} \texttt{olivia.caramello@uninsubria.it}}

\vspace{0.2cm}

{\small \textsc{Istituto Grothendieck, Corso Statuto 24,
		12084 Mondovì, Italy.}\\
	\emph{E-mail address:} \texttt{olivia.caramello@igrothendieck.org}}

\vspace{0.2cm}

The author is also affiliated to the GNSAGA group of the \emph{Istituto Nazionale di Alta Matematica}.
\end{document}